\documentclass[preprint,onefignum,onetabnum]{siamart190516}

\usepackage{color} 
\usepackage{graphicx} 
\usepackage{xspace}
\usepackage{amsmath} 
\usepackage{amsfonts} 
\usepackage{amssymb}
\usepackage{hyperref} 
\usepackage{amsbsy} 
\usepackage{ulem}
\usepackage{tikz} 
\usepackage[percent]{overpic} 

\usepackage{physics}

\newtheorem{lem}{Lemma}
\newtheorem{thm}{Theorem}
\newtheorem{cor}{Corollary}

\newtheorem{defn}{Definition}

\newcommand{\secref}[1]{Section~\ref{#1}}
\newcommand{\figref}[1]{Figure~\ref{#1}}
\newcommand{\defref}[1]{Definition~\ref{#1}}

\newcommand{\thmref}[1]{Theorem~\ref{#1}}
\newcommand{\lemref}[1]{Lemma~\ref{#1}}
\newcommand{\corref}[1]{Corollary~\ref{#1}}
\newcommand{\appref}[1]{Appendix~\ref{#1}} 

\newcommand{\reals}{{\mathcal{R}}\xspace} 
\newcommand{\nats}{{\mathcal{N}}\xspace} 

\newcommand{\eg}{{\it e.g.}\xspace}          
\newcommand{\ie}{{\it i.e.}\xspace}          
\newcommand{\etc}{{\it etc.\@}\xspace}       

\newcommand{\Order}[1]{O\!\left({#1}\right)}

\newcommand{\llVert}{\left| \left| \left|} 
\newcommand{\rrVert}{\right| \right| \right|} 

\newcommand{\bChar}[1]{\textbf{{#1}}} 

\newcommand{\DGIT}{{DGiT}\xspace}  
\newcommand{\tf}{{t_f}\xspace}

\newcommand{\taugamma}{{\tau_{\Gamma}}\xspace} 

\newcommand{\Fi}{{F_i}\xspace}  
\newcommand{\FOne}{{F_1}\xspace}  
\newcommand{\FTwo}{{F_2}\xspace}  

\newcommand{\tw}{{\tilde{w}}\xspace}

\newcommand{\hu}{{\hat{u}}\xspace} 
\newcommand{\hw}{{\hat{w}}\xspace} 
\newcommand{\whF}{{\widehat{F}}\xspace} 
\newcommand{\whL}{{\widehat{\Lambda}}\xspace}

\ifpdf
\hypersetup{
  pdftitle={A Multirate Discontinuous-Galerkin-in-Time Framework for Interface-Coupled Problems},
  pdfauthor={J.C. Connors and K.C. Sockwell}
}
\fi

\headers{A Multirate Discontinuous-Galerkin-in-Time Framework}{J.C. Connors, K.C. Sockwell}

\title{A Multirate Discontinuous-Galerkin-in-Time Framework for Interface-Coupled Problems\thanks{Submitted to the editors DATEXXX.
\funding{This material is based upon work supported by the U.S. Department of Energy, Office of Science, Office of Advanced Scientific Computing Research under Award Number DE-SC-0000230927 and the Coupling Approaches for Next-Generation Architectures (CANGA) project, a joint effort under the Scientific Discovery through Advanced Computing (SciDAC). Sandia National Laboratories is a multimission laboratory managed and operated by National Technology and Engineering Solutions of Sandia, LLC., a wholly owned subsidiary of Honeywell International, Inc., for the U.S. Department of Energy's National Nuclear Security Administration under contract DE-NA-0003525. This paper describes objective technical results and analysis. Any subjective views or opinions that might be expressed in the paper do not necessarily represent the views of the U.S. Department of Energy or the United States Government. SAND2021-15590 O }}}

\author{Jeffrey M. Connors\thanks{University of Connecticut, Storrs CT  
  (\email{jeffrey.connors@uconn.edu}, \url{https://www2.math.uconn.edu/\string~connors/}).}
\and K. Chad Sockwell\thanks{Center for Computing Research, Sandia National Laboratories, MS-1320, Albuquerque, NM 87185-1320, USA
  (\email{kcsockw@sandia.gov}).}}

\ifpdf
\hypersetup{
  pdftitle={An Example Article},
  pdfauthor={D. Doe, P. T. Frank, and J. E. Smith}
}
\fi

\begin{document}

\maketitle

\begin{abstract}
A framework is presented to design multirate time stepping algorithms for two dissipative models 
with coupling across a physical interface.  The coupling takes the form of boundary conditions 
imposed on the interface, relating the solution variables for both models to each other.  
The multirate aspect arises when numerical time integration is performed with different time step 
sizes for the component models.  In this paper, we seek to identify a unified approach 
to develop multirate algorithms for  these coupled problems. This effort is pursued though the use of  discontinuous-Galerkin time stepping methods, acting as a general unified framework, with different time step sizes.  The subproblems  are coupled across  user-defined intervals of time, called {\it coupling windows}, using polynomials 
that are continuous on the window.  The coupling method is shown to reproduce the correct 
interfacial energy dissipation, discrete conservation of fluxes, and asymptotic accuracy. In principle, 
methods of arbitrary order are possible.  As a first step, herein we focus on the presentation and 
analysis of monolithic methods for advection-diffusion models coupled via generalized 
Robin-type conditions.  
The monolithic methods could be computed using a 
Schur-complement approach. We conclude with some discussion of future developments, such as different interface conditions and partitioned methods. 
\end{abstract}

\begin{keywords}
 Galerkin in Time, Coupled-problem stability, interface coupling, multirate methods
\end{keywords}

\begin{AMS}
  65L60,34D20,65M12,65,35
\end{AMS}

\section{Introduction}\label{sec:intro} 
Multiphysics simulations that are decomposed into subproblems, each with their own domain, strongly benefit from unique and specialized treatment of each component, leading to heterogeneous numerical methods.  
If these models exhibit different time 
scales, it is desirable for efficiency to optimize time integrators using 
independent step sizes for each model.  Stable and accurate numerical coupling across subproblem interfaces,
while retaining differently sized time steps and underlying properties of the governing system, remains challenging 
for a general set of problems and interface conditions.  To meet this challenge, there is a need to design new types of
multirate time-stepping methods centered around 
coupling aspects.  Recently, there is 
increasing attention toward the development of multirate algorithms for 
interface-coupled problems, such as  
coupled heat equations~\cite{CHL2009SINUM,MonBir2019}, 
fluid-fluid interaction~\cite{CH2012,CD2019}, 
fluid-structure interaction~\cite{RuthEtAl2020}, 
Stokes-Darcy~\cite{HL2021,LiHou2018,RYBAK2014327,SZL2013}, 
Stokes-Richards~\cite{RybakEtAl2015}, 
Stokes-Darcy-transport~\cite{ZRC2019}, and
dual-porosity-Stokes~\cite{WANG2021265} coupled flows.  More broadly, 
large-scale multiphysics applications may involve many 
interface-coupled components, as exemplified by the Energy Exascale Earth Systems Model (E3SM)~\cite{E3SM}.   

Within the context of multirate interface-coupled problems, problems that benefit from discretizing with differently-sized time steps on each side of the interface, the mismatching of the step sizes \(\Delta t_i\) creates a fundamental issue when coupling two sub-problems together. Each subproblem requires a flux \(\Fi\) which couples the subproblems through an interface at each respective time step, but the calculation of the flux is not clear. Beyond 
accuracy there can be additional considerations like stability, conservation of fluxes between subdomains, or model sensitivity (to name a few), further complicating matters with each additional concern.    

Although the continuous model equations~\eqref{eq:pdes}-\eqref{eq:ics} possess a flux that is defined at all points in time, the multirate discretization generates issues in calculating \(\Fi\).  An elegant solution is to recover the continuity of the flux in time, in some sense. To do this, we propose to define the flux as a new auxiliary variable, with~\eqref{eq:int_bcs}  to be treated as an additional auxiliary equation on a time-scale that is potentially different from either subproblem, associated instead with the coupling window size, \(\Delta t\). This leads to a discrete monolithic system with two subproblems and an auxiliary interfacial problem, all possibly defined on different time scales. The goal is to solve the auxiliary problem in such a way that it  connects the fine time scales of the subproblems together, essentially gluing the two subproblems together over the coarser scale of the coupling window.  This puts the problem into the form of a system of differential algebraic equations (DAEs), in which the auxiliary equation is an algebraic condition in time. In following the solution methods for DAE formulations with coupled problems~\cite{peterson2019explicit, sockwell2020interface,gravouil2015heterogeneous,AscPet1998}, the new auxiliary problem could then be solved before the individual subproblems to determine the fluxes over the coupling window. Then each subproblem can utilize the flux, which is now defined at all substeps. 

To accomplish this, we \textit{equip the flux with a variational structure in time, as opposed to being defined only at some set of points in time}.  This stands in contrast with methods (\eg~\cite{gravouil2015heterogeneous}) that define the interface information on the super-set of time steps, which constricts the choice of the mathematical definition of the flux too much for our goals.  
Although such a strategy could be taken for specific methods and situations, we wish to pursue a general framework for the description and analysis of a large class of methods. We believe such a broad framework will allow some given methods to be used for both sub-problems and then glue them together, regardless of their form or time step sizes.  This is as opposed to a top-down approach that specifies a coupling methodology for a specific time-integrator, but which may limit freedom of choice. 

The desire to prescribe a variational structure in time for the flux, and for a large class of time-integrators, motivates the use of a powerful framework, known as the \textit{Discontinuous Galerkin in Time} (DGiT) framework. The DGiT framework is simply a discontinuous-Galerkin method in the one dimension of time. The utility of the this framework is in its (piecewise) polynomial description of  the time evolution for a system.  Many, and almost all well-known time-integrators fall into this framework. The variational structure allows for quantities, in our the case the flux \(\Fi\) and the solution \(u_i\), to have a continuous form almost everywhere over the time domain; something not considered in typical descriptions of time-integrators. Figure~\ref{fig:time_steps} illustrates the time representation of variables. 
Standard within the~\DGIT framework, states $u_i$ are approximated as polynomials $u_i^n$ on subintervals 
$(t_i^{n-1} , t_i^n)$.  Jumps are allowed between subintervals, so the state values at times $t_i^n$ are denoted 
by $U_i^n$.  However, the fluxes $F_i$ are approximated as a single, continuous polynomial function $F_i^{\tilde{n}}$ on  coupling window number $\tilde{n}$.  These representations are ideal to make rigorous statements about numerical 
coupling properties.

\subsection{Concrete Example}
Coupled, evolutionary diffusion~\cite{CHL2009SINUM,LBD2015,ZHANG2020,ZHENG20085272}  
or advection-diffusion equations~\cite{ERKMEN2017180,PETERSON2019459},   
have been discussed when studying the numerical properties of coupling algorithms, sometimes as proxies for applications.  
We proceed to specify an advection-diffusion model problem used to present the key ideas of the 
multirate coupling framework.  Denote the real and natural number systems 
by $\reals$ and $\nats$, respectively.  Bold font is reserved for vectors.  The type of 
vector (coordinate, column, \etc) may be inferred from the context.  
We reserve $t$ for time, with coordinates in space denoted by 
$\bChar{x}=(x_1 , \ldots , x_d)\in\reals^d$.  

Let $\Gamma = (0,1)^{d-1} \times \{ 0\}\subset \reals^d$, 
choosing $d=2$ or $d=3$.  We consider two open domains, 
$\Omega_1 = \Gamma \times (0,1)$ and 
$\Omega_2 = \Gamma \times (-1,0)$, so their common interface is 
$\overline{\Gamma} = \partial \Omega_1 \cap \partial \Omega_2$.  
The exterior boundary components are  
$\Gamma_i =\partial \Omega_i \setminus \Gamma$, for $i=1,2$. 
Let $\bChar{n}_i$ denote the unit outward-pointing normal vector on 
the boundary of $\Omega_i$.  
Then consider the problem to solve for 
$u_i = u_i(\bChar{x} ,t)$, $i=1,2$, where 
$u_i:\overline{\Omega_i}\times [0,\tf] \to \reals$, 
$\tf \in \reals^+$, satisfies 
\begin{alignat}{2}
  \dot{u}_i &= \nabla \cdot (\nu_i \nabla u_i - \bChar{s}_i u_i) + f_i\quad &\text{on}\;\Omega_i \times (0,\tf ]\;,\label{eq:pdes} \\ 
 u_i &=0 \quad &\text{on}\;\Gamma_i \times (0,\tf ]\;,\label{eq:ext_bcs} \\ 
 -\nu_i \bChar{n}_i \cdot \nabla u_i  
&= b_{i,1} u_1 +b_{i,2} u_2 -g_i \quad &\text{on}\;\Gamma \times (0,\tf ]\;, \label{eq:int_bcs} \\ 
 u_i (\bChar{x},0) &=u_i^0 (\bChar{x}) \quad &\text{on}\;\Omega_i ,\label{eq:ics}
\end{alignat} 
Here, $\dot{u}_i$ denotes time differentiation, $b_{i,j}$ and \(\nu_i >0 \) are constants, 
\( \bChar{s}_i :\overline{\Omega_i}\to \reals^d \) are steady 
advection fields, \(f_i :\Omega_i\times (0,\tf]\to\reals \) are body forcing functions, 
and $g_i:\Gamma \times (0,\tf]\to\reals$ are interfacial forcings.  
Note that $\bChar{n}_i \cdot\bChar{s}_i$ does not appear in~\eqref{eq:int_bcs}.  
We only consider fields $\bChar{s}_i$ satisfying  $\bChar{n}_i \cdot\bChar{s}_i =0$ on $\partial \Omega_i$, 
for simplicity.  

Also, we introduce a more compact notation that emphasizes the role of interfacial fluxes 
later in coupling method.  Denote these fluxes by $\Fi$, where 
\begin{equation}
    \Fi := b_{i,1} u_1 +b_{i,2} u_2 -g_i, \quad i=1,2.  
    \label{eq:Fi}  
\end{equation} 
The interface conditions are generalized Robin-type boundary conditions.  

Let us illustrate some key technical issues for multirate coupling that we seek 
to address.  To this end, consider time stepping with different step sizes, $\Delta t_i$, for 
each subdomain, $\Omega_i$.  In fact, we shall assume that the discrete time levels coincide at regular 
intervals, for which purpose we reserve the index $\tilde{n}=0,1,\ldots, N$ and 
provide the following definition. 
\begin{defn}[Synchronization times]\label{defn:sync_points} 
Given $N\in\nats$, we set $\Delta t :=\tf/N$ and define {\it synchronization times} in time 
as $t^{\tilde{n}}=\tilde{n}\Delta t$, for $\tilde{n}=0,1,\ldots, N$.  
\end{defn}
The calculations for both subdomains will be coupled together between any two successive synchronization
times, motivating the next definition.  
\begin{defn}[Coupling window]\label{defn:cw}
The interval $[t^{\tilde{n}-1} , t^{\tilde{n}}]$ is referred to as {\it coupling window} $\tilde{n}$, for 
each $\tilde{n}=1,2,\ldots , N$.  
\end{defn}  
The number of substeps taken on $\Omega_i$ over a coupling window is 
denoted by $M_i\in\nats$, which we fix to be independent of the window.  No relationship is 
required between $M_1$ and $M_2$, but we assume $M_i \Delta t_i =\Delta t$ for $i=1,2$.  
A local time level index $n$ (rather than $\tilde{n}$) is used for substeps on a given coupling window.  
\begin{defn}[Time substeps]\label{defn:substeps} 
Relative to coupling window $\tilde{n}$, $1\leq \tilde{n}\leq N$ , we define the local time 
level notation for subdomain $\Omega_i$ as 
$t_i^n = t_i^n (\tilde{n}) :=  t^{\tilde{n}-1} +n\Delta t_i$ for any integer $n$.   
\end{defn}
Note that $t_i^0 = t^{\tilde{n}-1}$ and $t_i^{M_i} = t^{\tilde{n}}$.  

Spatial discretizations and also forcings $f_i$ and $g_i$ may be ignored 
at this point, but for clarity here we consider the specific example of using 
Crank-Nicolson time stepping for both subdomains.   
On coupling window $\tilde{n}$, let $u_i^n \approx u_i (t_i^n)$ for each $n$, with 
$u_i^{n-1/2}:= (u_i^n +u_i^{n-1} )/2$ for compactness of notation.  Each $u_i^n$ is determined 
via 
\begin{equation}
\frac{1}{\Delta t_i} \left( u_i^n - u_i^{n-1} \right) 
= \nabla \cdot \left( \nu_i \nabla u_i^{n-1/2} - \bChar{s}_i u_i^{n-1/2} \right) . 
    \label{eq:CN_example1} 
\end{equation}
Our focus now is on the multirate treatment of the interface conditions~\eqref{eq:int_bcs}.  
These boundary conditions are implemented in the form 
\begin{equation}
    -\nu_i \bChar{n}_n \cdot \nabla u_i^{n-1/2} 
    = \Fi^{n-1/2} , 
    \label{eq:CN_example2}  
\end{equation}
where the numerical flux $\Fi^{n-1/2}$ is some approximation for~\eqref{eq:Fi} 
(with $g_i=0$ for now).  Interestingly, one cannot even write down 
\[
\Fi^{n-1/2} = b_{i,1} u_1^{n-1/2} + b_{i,2} u_2^{n-1/2} \quad \text{(low-order accuracy)} 
\]
because the data $u_1^{n-1/2}$ and $u_2^{n-1/2}$ are not defined at the same times
in the multirate context; indeed, from~\defref{defn:substeps} we have 
\[
t_1^{n-1/2} - t_2^{n-1/2} 
= (n-1/2)\left( \Delta t_1 - \Delta t_2 \right) \neq 0 .  
\] 

\begin{figure}
\centering 
\includegraphics[scale=0.35]{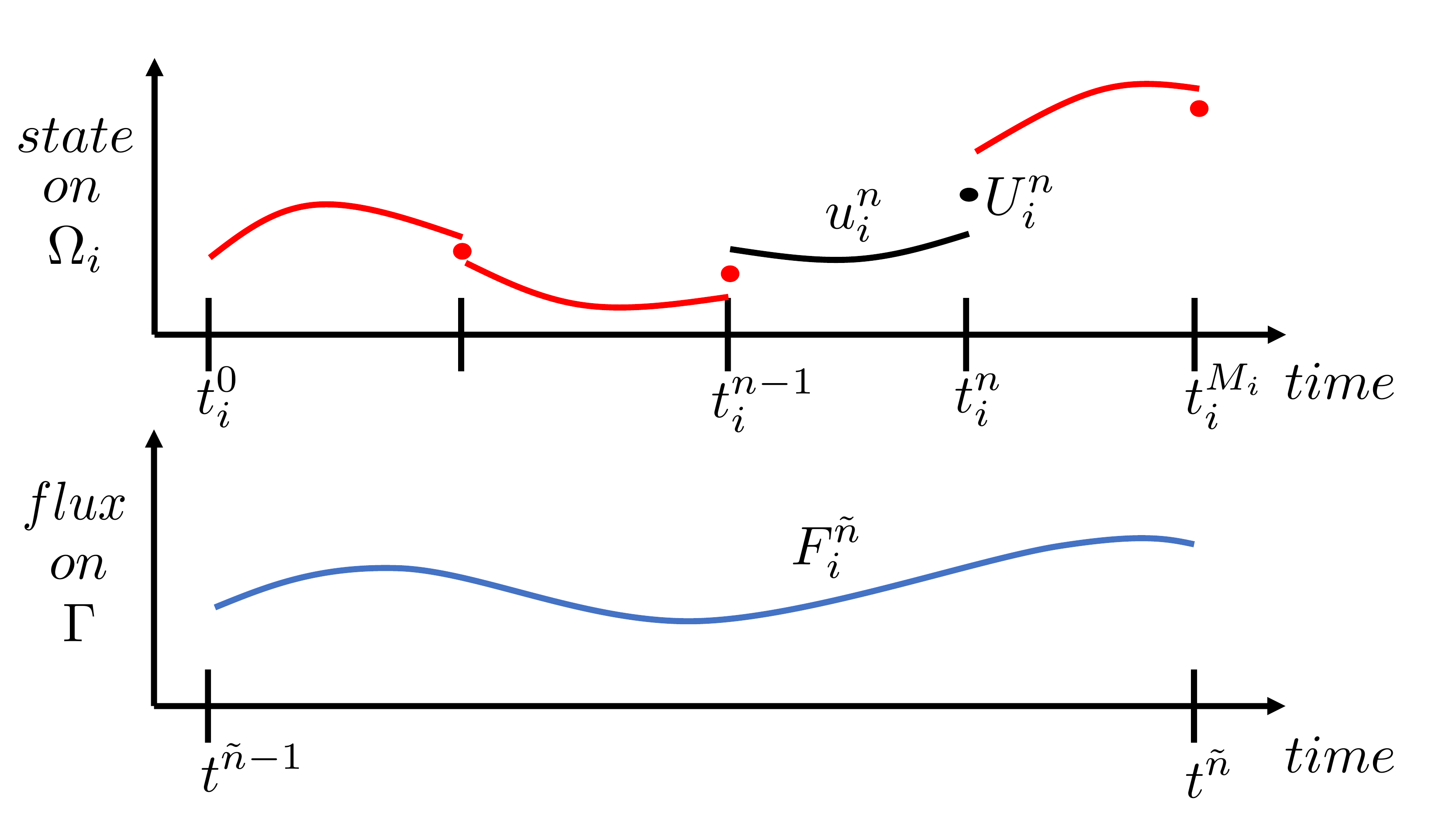} 
\caption{Multirate time levels on a coupling window $[t^{\tilde{n}-1} , t^{\tilde{n}}]$ for the computations on $\Omega_i$ (top), for $i=1,2$.  The states on $\Omega_i$ are approximated as $U_i^n$ at the times $t_i^n$ for $n=0,1,\ldots ,M_i$.  In between times $t_i^{n-1}$ and $t_i^n$ the approximation is a polynomial, $u_i^n$.  This data is determined by the~\DGIT representation of a time integrator.  Traces along $\Gamma$ are projected into a coarse-scale polynomial space 
on the coupling window.  Fluxes are computed by querying a polynomial representation $F_i^{\tilde{n}}$ on the coupling window (bottom).}\label{fig:time_steps}
\end{figure}

An example algorithm within the framework can be demonstrated using the Crank-Nicolson discretization 
in~\eqref{eq:CN_example1}, with falls under the~\DGIT class of methods~\cite{DelDub1986}.  The details of 
the~\DGIT methods are deferred until~\secref{sec:DGITtime}.  We merely wish to spell our the main idea here.
Consider the system consisting of the two subproblems within some coupling window, defined at the times \(t_i^n\) in correspondence to Figure~\ref{fig:time_steps}
\begin{align}
\frac{1}{\Delta t_1} \left( U_1^n - U_1^{n-1} \right) 
= \nabla \cdot \left( \nu_1 \nabla u_1^n \left( t_1^{n-1/2}\right) - \bChar{s}_1 u_1^n \left( t_1^{n-1/2}\right) \right)\;\text{on}\;\Omega_1\;,
    \label{eq:CN_mono1} 
    \\
\frac{1}{\Delta t_2} \left( U_2^n - U_2^{n-1} \right) 
= \nabla \cdot \left( \nu_2 \nabla u_2^n \left( t_2^{n-1/2}\right)  - \bChar{s}_2 u_2^n \left( t_2^{n-1/2}\right) \right)\;\text{on}\;\Omega_2 \;,
    \label{eq:CN_mono2} 
\end{align}
where \(u_i^k (t_i^k) = U_i^k\) is specified for Crank-Nicolson.   The variational structure realized in the DGiT framework (omitted here) leads to a polynomial representation of the solution in time interval \( (t_i^{n-1},t_i^n ) \). In this case, the polynomial form can be represented via the Lagrange interpolant 
\[
u_i^n(t) = U_i^{n-1} \frac{t_i^{n}-t}{\Delta t_i}+U_i^{n}\frac{t-t_i^{n-1}}{\Delta t_i}  \;, \quad 1\leq n \leq M_i, \ i=1,2 .
\]
We now wish to construct a time-continuous polynomial of the flux over the coupling window \([t^{\tilde{n}-1},t^{\tilde{n}}]\) corresponding to Figure~\ref{fig:time_steps}. 

In this paper, we investigate a least-squares approach to compute the fluxes.  In the Crank-Nicolson example, it means that we choose a polynomial order, say $r_i$, to approximate the states on 
the coupling window, and then 
require that  
\begin{equation}
\int_{t^{\tilde{n}-1}}^{t^{\tilde{n}}} F_i^{\tilde{n}}\, p\, dt =\int_{t^{\tilde{n}-1}}^{t^{\tilde{n}}} \left(b_{i,1} u^{\tilde{n}}_{\Gamma ,1} + b_{i,2} u^{\tilde{n}}_{\Gamma ,2}\right) \, p \,dt,  \quad \forall p\in \mathbb{P}_{r_i} \left( t^{\tilde{n}-1} ,t^{\tilde{n}} \right) ,
    \label{eq:Fi_LS_CN}
\end{equation} 
where the auxiliary variables $u^{\tilde{n}}_{\Gamma , i}\in \mathbb{P}_{r_i} \left( t^{\tilde{n}-1} ,t^{\tilde{n}} \right)$ are calculated 
from 
\begin{multline}
\int_{t^{\tilde{n}-1}}^{t^{\tilde{n}}} u^{\tilde{n}}_{\Gamma ,i} \, p \, dt 
= \Delta t_i \sum_{n=1}^{M_i} u_i^n |_{\Gamma} \left( t_i^{n-1/2} \right) \, p\left( t_i^{n-1/2} \right) \\ 
\approx 
\sum_{n=1}^{M_i} \int_{t^{n-1}}^{t^n} u_i^n |_{\Gamma}\, p \, dt ,\quad \forall p\in \mathbb{P}_{r_i} \left( t^{\tilde{n}-1} ,t^{\tilde{n}} \right) . 
    \label{eq:ui_LS_CN}
\end{multline} 
Here, $u_i^n |_{\Gamma}$ is the restriction of $u_i^n$ to $\Gamma$. This condition can be interpreted 
either as a least-squares fit of the piecewise-polynomial flux 
$b_{i,1} u_1^n |_{\Gamma} + b_{i,2} u_2^n |_{\Gamma}$ or, equivalently, as an $L^2$-projection 
into the smoother space of polynomials of order $r_i$ on the coupling window, with an added quadrature approximation.  The reason for the quadrature is explained later, as it is technical, but it is related to the derivation of Crank-Nicolson as a~\DGIT method.

In fact, we show later that this approach naturally enables the weak enforcement of flux conservation, the 
correct interfacial energy dissipation, and high-order consistency, all  
simultaneously.  The $L^2$-projection couples the data for all substeps on a window, 
so they would need to be computed together.  In practice, this might be possible 
using dual Schur-complement technique of the monolithic system~\cite{prakash2004feti,prakash2014computationally,gravouil2015heterogeneous,peterson2019explicit,sockwell2020interface}, that isolate the flux at the interface \(\Gamma\) that can then be utilized to construct a partitioned solve of the two subproblems.  We 
therefore refer to the methods as monolithic.  However, our goals go beyond 
this for future work, where we will discuss partitioned methods.  In this paper, 
we seek only to provide the monolithic framework and analysis of existence 
and convergence (including convergence rates) as a first step, which still illustrates 
the most critical guiding principles.  

Next, we provide some background for interface-coupled problems. The work in this paper is laid out as follows.  In~\secref{sec:semidiscrete}, we 
provide a semidiscrete model, which is referenced afterwards to present and analyze the coupling 
framework.  The \DGIT time-discretization is detailed in~\secref{sec:DGITtime}. The framework itself is presented in~\secref{sec:multirate}.  
The numerical analysis is then given 
in~\secref{sec:analysis}.  The results are summarized 
in~\secref{sec:summary}, along with some discussion of future 
steps for the development of the proposed framework.  


\section{Semi-discrete formulation}\label{sec:semidiscrete}
The~\DGIT framework we study begins with a method-of-lines approach 
to first discretize~\eqref{eq:pdes}-\eqref{eq:ics} in space and 
derive a semi-discrete approximation.  In principle, the choice 
of numerical method for this step is not important, but we will 
apply a standard type of finite element method.  After choosing a 
particular method, we can make some statements about the 
properties of the semi-discrete system to help motivate the 
proposed~\DGIT framework.  

Given that $d=2$ ($d=3$), let $E_{i,j}$ be triangles (tetrahedra) or 
quadrilaterals (parallelipipeds), for $1\leq j \leq N_i$ and $i=1,2$, 
such that the sets $\tau_i=\{ E_{i,j} \}_{j=1}^{N_i}$ are 
conforming, regular meshes on $\Omega_i$.  In the former case, 
denote the spaces of polynomials of order at most $k$ on any real 
set $S\subset\reals^d$ by $\mathbb{P}_{k} (S)$.  Otherwise 
$\mathbb{P}_{k} (S)$ will be tensor-product polynomials with terms 
up to order $k$ in each coordinate.  Given an integer $k\geq 1$, we  
define two (for $i=1,2$) conforming finite element spaces 
\[
{U}_{i} = \left\{ 
v\in \mathcal{C}\left( \overline{\Omega_i}, \reals\right)\, : \, 
v|_{E_{i,j}} \in \mathbb{P}_{k} (E_{i,j}), \ 
\forall E_{i,j}\in \tau_i \ \text{and} \ 
v|_{\Gamma_i} =0
\right\} . 
\]
On $\Gamma$ we assume that the meshes $\tau_i$ coincide and 
form a regular partition of the interface into 
segments ($d=2$) or either triangles or quadrilaterals ($d=3$).  
Denote this partition by $\taugamma=\{ E_{\Gamma ,j} \}_{j=1}^{N_\Gamma}$.  The implied 
finite element ``trace space'' on $\Gamma$ is  
\[
{U}_\Gamma = \left\{ 
\mu \in \mathcal{C}\left( \overline{\Gamma} , \reals\right)\, : \, 
\mu|_{E_{\Gamma ,j}} \in \mathbb{P}_{k} (E_{\Gamma ,j}), \ 
\forall E_{\Gamma ,j}\in \taugamma , 
\ \text{and} \ \mu|_{\partial \Gamma}=0
\right\} .
\]

The dimensions of the spaces ${U}_i$ and ${U}_\Gamma$ are denoted hereafter 
by $d_{\Omega_i}$ and $d_{\Gamma}$, respectively.   
Let $\{ \phi_{i,j} \}_{j=1}^{d_{\Omega_i}}$, $i=1,2$, be nodal  
bases used for the spaces ${U}_i$.  Similarly, 
let $\{ \mu_{j} \}_{j=1}^{d_{\Gamma}}$, be the 
induced nodal basis of the space ${U}_\Gamma$ found by 
taking traces.  Some bilinear forms are required to derive our 
semi-discrete system.  We define for $i=1,2$ 
\begin{align*}
    (v_i , w_i) &:= \int_{\Omega_i} v_i\, w_i\, d\bChar{x} , \quad \forall v_i ,\, w_i\in L^2 (\Omega_i), \\
    a(v_i , w_i) &:= \int_{\Omega_i} \nu_i \nabla v_i\cdot \nabla w_i\, d\bChar{x} , \quad \forall v_i ,\, w_i\in H^1 (\Omega_i), \\ 
    b(v_i , w_i) &:= \int_{\Omega_i} \nabla\cdot (\bChar{s}_i v_i) w_i\, d\bChar{x} , \quad \forall v_i \in H^1 (\Omega_i),\, w_i\in L^2 (\Omega_i).  
\end{align*}
The advection and diffusion terms
are grouped notationally by defining 
\[
L (u_i , v_i):= a (u_i , v_i)+ b(u_i , v_i) . 
\] 
In relation to the interface we define 
\[
(\mu , \lambda )_\Gamma := \int_\Gamma \mu\, \lambda \, d\Gamma , 
\quad \forall \mu ,\, \lambda \in L^2 (\Gamma) . 
\] 
The method would require stabilization for certain advection 
fields, but we do not need to include such cases in this paper.  
Enough strong advection effects can still be studied under the 
stable case that $\nabla \cdot \bChar{s}_i$ is not too large, for 
example.  

In the semi-discrete mixed formulation, we seek to find 
$u_i :[0,\tf]\to U_i$ 
and $F_{i}:(0,\tf]\to U_{\Gamma}$ 
such that 
\begin{align}
    (\dot{u}_i , v_i) &= -L (u_i , v_i) - (F_i,v_i)_\Gamma + (f_i , v_i) , \label{eq:semi_1a} \\ 
&\qquad \qquad \qquad \forall v_i \in U_i,\ 0<t\leq \tf,\ i=1,2, \nonumber \\ 
    (u_i (0), v_i) &= (u_i^0,v_i) ,\ \forall v_i\in U_i, \ i=1,2, \label{eq:semi_1b} 
\end{align} 
with the interface conditions weakly enforced via 
\begin{equation}
(F_i , \mu_i )_\Gamma 
=(b_{i,1} u_1 + b_{i,2} u_2 -g_i , \mu_i )_\Gamma , \ \forall \mu_i \in U_\Gamma , \ i=1,2. 
\label{eq:semi_1d} 
\end{equation}

\subsection{Semi-Discrete Differential Algebraic Equation (DAE) Equivalence}\label{sec:DAEschur}
The semi-discrete problem~\eqref{eq:semi_1a}-\eqref{eq:semi_1d} may be expressed 
equivalently as a system of DAEs, where the 
coupling conditions~\eqref{eq:semi_1d} correspond to the algebraic constraints. The Robin condition in~\eqref{eq:semi_1d} gives rise to an index-1 Hessenberg form DAE~\cite{AscPet1998}, as opposed to the more complex index-2 situation seen in a Lagrange mutlipler setting~\cite{peterson2019explicit}. The DAE is trivial in some sense because it leads directly to an ODE through substitution of the interface condition in the model equations. Although we could work directly with the ODE, retaining the structure of a DAE lends itself to the development of coupling algorithms~\cite{sockwell2020interface}. The algebraic constraints can be eliminated through the enforcement of a Neumann condition, leading to an equivalent set of ODE's with the variational form 
\begin{align*}
    \sum_{i=1,2}(\dot{u}_i , v_i) &= -\sum_{i=1,2} L (u_i , v_i) + \mathcal{F}( (u_1 , u_2) , (v_1 , v_2) )  ,  \\ 
&\qquad \qquad \qquad \forall (v_1, v_2) \in U_1 \times U_2 ,\ 0<t\leq \tf,\, \nonumber \\ 
    (u_i (0), v_i) &= (u_i^0,v_i) ,\ \forall v_i\in U_i, \ i=1,2, 
\end{align*} 
where $\mathcal{F}$ is a forcing term and contains
the fluxes \(F_{i}\) obtained from the algebraic constraint. Under mild 
conditions on the problem data, the well-posedness of the semi-discrete problem follows 
from classical results for ODEs;~\eg~\cite{Cara1927}.  We skip the details, for brevity. This structure can be leveraged to design coupling algorithms through FETI or dual-Schur complement based domain decomposition methods~\cite{farhat1994dual} that construct an interface flux in a manner similar to variational flux recovery~\cite{carey1985approximate}. In a sense, the details of the ODE--DAE equivalence can be exploited to develop coupling algorithms through the implicit function of the algebraic constraint.  In~\cite{prakash2004feti,prakash2014computationally,gravouil2015heterogeneous,peterson2019explicit,sockwell2020interface}, the so-called dual-Schur complement methods that isolates the fluxes \(\Fi\) at the interface \(\Gamma\) leading to ``implicitly-coupled'' algorithms that are consistent with the underlying implicit discretization. We simply wish to note this relation to identify possible solution algorithms to be appliedwith our monolithic framework.

\section{\DGIT Discretization}\label{sec:DGITtime}
In this section we seek only to provide the~\DGIT description for  
the discretization of~\eqref{eq:semi_1a} on a single substep in 
time.  The details of the multirate flux computations over a 
coupling window are left to~\secref{sec:multirate}.  In essence, the \DGIT framework employs a discontinuous Galerkin (DG) method in the one dimension of time. Methods developed in this framework are characterized by the classical constituents of a DG method: a test and trial basis, quadrature, and boundary conditions (referred to as side conditions; see below).  The time 
interval to compute on subdomain $\Omega_i$ for a given 
coupling window is 
\[
I_i^n := \left( t_i^{n-1} , t_i^n \right) . 
\] 
On this interval, the state $u_i$ is approximated as a polynomial 
(in time), denoted by $u_i^n$.  Recall that the finite element basis is $\{ \phi_{i,j} \}_{j=1}^{d_{\Omega_i}}$.  We choose a 
polynomial order $q$ and seek $u_i^n$ in the approximation space 
\begin{equation}
    \mathcal{U}_i^n (q)  := 
    \left\{ 
     v = \sum_{j=1}^{d_{\Omega_i}} v_j (t) \phi_{i,j} (\bChar{x}) \ : \ v_j \in \mathbb{P}_q \left( I_i^n \right), \ 1\leq j \leq d_{\Omega_i} 
    \right\} .
    \label{eq:U_i_n}
\end{equation} 
At the time $t_i^n$, the~\DGIT methods allow the state 
approximation to be discontinuous.  In addition to $u_i^n$, 
a {\it side value} approximation, denoted by $U_i^n$, is also 
computed to represent the state at time $t_i^n$ 
(see~\figref{fig:time_steps}).  

The constraint equations needed to determine $u_i^n$ and 
$U_i^n$ can take two forms.  First, the~\DGIT methods allow 
one to specify $n_s$ linear {\it side conditions} that relate 
values of $u_i^n$ at some chosen points in time to some subset 
of the side values.  Some methods do not require any side conditions, 
in which case $n_s=0$.  If $n_s\geq 1$, we assume 
the side conditions to be specified in the following way.  Let 
$-\infty < \theta_1 < ... < \theta_{n_s} \leq 1$.  In order to 
construct $k_s$-step methods, one reaches back in time for side 
values.  Then the side conditions take the form 
\begin{equation}
    u_i^{n} \left( t_i^{n-1} +\theta_k \Delta t_i \right)  
    =  \sum_{l=1}^{k_s+1} \left(\mathbb{D}\right)_{k,l} U_i^{n+1-l} , \ 
    1\leq k \leq n_s, 
    \label{eq:side_conditions} 
\end{equation} 
where the coefficients $(\mathbb{D})_{k,l}$ form the entries 
of a fixed matrix $\mathbb{D}$ of size $n_s \times (k_s +1)$.  
We note here that although the functions $u_i^n$ are 
defined in $\mathcal{U}_i^n (q)$ with time domain $I_i^n$, the 
side conditions may refer to times outside of this interval.  
Since $u_i^n$ is just a polynomial in time, the domain can be 
extended trivially to $\reals$; this is the meaning of 
$u_i^n (t)$ whenever $t$ is outside of $I_i^n$.  
However, $u_i^n (t) \neq u_i^m (t)$ 
in general if $n\neq m$.  

The remaining constraint equations are variational, and 
weakly enforce~\eqref{eq:semi_1a} in time on the interval 
$I_i^n$.  
These are 
\begin{align}
   \left( U_i^n, v_i^n \left(t_i^n\right) \right) -\int_{I_i^n} \left( u_i^n , \dot{v}^n_i \right) \, dt &= 
   \left( U_i^{n-1}, v_i^n \left(t_i^{n-1}\right) \right) \label{eq:DG_1a} \\ 
   -\int_{I_i^n} L\left( u_i^n, v_i^n \right)  +  \left( F_i^{\tilde{n}} , v_i^n \right)_\Gamma &-( f_i ,v_i^n )\, dt, \ \forall v_i^n \in \mathcal{U}_i^n (q+1-{n_s}). \nonumber 
\end{align}
The maximum order is $q+1-n_s$ for the test functions $v_i^n$ 
in~\eqref{eq:DG_1a} because there are already $n_s$ side 
conditions.  Together, we have $q+2$ implied constraints in 
time to solve for $u_i^n \in\mathcal{U}_i^n (q)$ (requiring 
$q+1$ such conditions) and $U_i^n$ (one more condition).  
The flux function $F_i^{\tilde{n}}$ is determined later, but 
it is a polynomial in time on $I_i^n$.  

Since the side conditions require some previous data,~\ie 
for multistep methods, a starting procedure or initial 
set of data would generally be needed.  We assume that 
data is somehow provided on coupling windows $\tilde{n}$ 
for $1\leq \tilde{n} \leq N_0$, for some appropriate value 
of $N_0$.  Then~\eqref{eq:side_conditions}-\eqref{eq:DG_1a} 
is only solved on windows $\tilde{n}\geq N_0$.  

We continue to illustrate the ideas using the Crank-Nicolson 
method.  As a~\DGIT method, it is actually continuous in 
time because one specifies $n_s=2$ side conditions 
\[
u_i^n \left( t_i^{n-1} \right) = U_i^{n-1} 
\quad \text{and} \quad 
u_i^n \left( t_i^{n} \right) = U_i^{n} .
\]
Thus, we have for each time $t_i^n$ that 
$u_i^{n+1} (t_i^n) = U_i^n = u_i^{n} (t_i^n)$.  Choose $q=1$ 
so that $u_i^n \in \mathcal{U}_i^n (1)$ is a first-order 
polynomial in time.  Then $q+1-n_s=0$, and~\eqref{eq:DG_1a} 
reduces to 
\begin{align}
   \left( U_i^n, v_i^n \right) &= 
   \left( U_i^{n-1}, v_i^n \right) 
   -\Delta t_i L \left( \frac{U_i^{n-1} +U_i^{n}}{2}, v_i^n \right) 
   \label{eq:DG_1a_CN} \\ 
   &-\int_{I_i^n}  \left( F_i^{\tilde{n}} , v_i^n \right)_\Gamma -( f_i ,v_i^n )\, dt, \ \forall v_i^n \in \mathcal{U}_i^n (0). \nonumber 
\end{align}
The remaining integral terms can be approximated using a 
quadrature rule to obtain the classical form of 
Crank-Nicolson.  We leave the method in the above form for 
later reference, and now move on to the details of 
multirate coupling.

\section{The monolithic multirate \DGIT coupling framework}\label{sec:multirate}
In this section, we present the multirate coupling framework that is built upon the \DGIT discretization in~\secref{sec:DGITtime}. We now consider treating both subproblems in \(\Omega_1\) and \(\Omega_2\) with \DGIT and letting the time step sizes vary in each domain. The information required for coupling, contained in the fluxes \(\Fi\), is then provided using a coarse-scale approximation over the coupling windows. We also consider the flux conservation and energy dissipation with the coarse-scale approximation of \(\Fi\). 

The substep computations described in~\secref{sec:DGITtime} are grouped together  
over a given coupling window, say $[t^{\tilde{n}-1} , t^{\tilde{n}}]$.  The 
interface conditions are enforced in a polynomial space 
$\mathbb{P}_{r_i} (I^{\tilde{n}})$ of chosen order $r_i$, where we define 
\[
I^{\tilde{n}} := \left( t^{\tilde{n}-1} , t^{\tilde{n}} \right) 
\]
as the interior of the coupling window because jumps are allowed between 
windows.  In fact, interface conditions need not be specified at the 
synchronization times $t^{\tilde{n}}$, because only the integrated data 
\[
\int_{I_i^n} \left( F_i^{\tilde{n}} , v_i^n \right)_\Gamma \, dt
\] 
appears in the substep calculations when using~\DGIT methods.  

The first step to determine interface conditions is to project the 
interface traces of the states, $u_i^n |_{\Gamma}$ for $1\leq n\leq M_i$, into the 
coarser-scale coupling space.  Precisely, this latter space is 
\begin{equation}
\mathcal{V}^{\tilde{n}} (r_i) := \left\{ 
\lambda = \sum_{j=1}^{d_\Gamma} \lambda_j (t) \mu_j (\bChar{x}) \ : \ 
\lambda_j \in \mathbb{P}_{r_i} (I^{\tilde{n}}),\ 1\leq j \leq d_{\Gamma} 
\right\} , 
    \label{eq:V_i_n} 
\end{equation}
recalling that $\{ \mu_j \}_{j=1}^{d_\Gamma}$ is a basis of the finite element 
trace space on $\Gamma$.  We now define auxiliary variables 
$u_{\Gamma ,i}^{\tilde{n}} \in \mathcal{V}^{\tilde{n}} (r_i)$ as the solutions to 
\begin{equation}
\int_{I^{\tilde{n}}} \left( u_{\Gamma ,i}^{\tilde{n}} ,\lambda^{\tilde{n}} \right)_\Gamma \, dt 
= 
\sum_{n=1}^{M_i} \int_{I_i^n} \left( u_{i}^n  ,\lambda^{\tilde{n}} \right)_\Gamma \, dt ,
\quad \forall \lambda^{\tilde{n}} \in \mathcal{V}^{\tilde{n}} (r_i) .
    \label{eq:DG_1c}  
\end{equation} 
The flux functions on the coupling window are now given by $F_i^{\tilde{n}} \in \mathcal{V}^{\tilde{n}} (r_i)$ 
satisfying 
\begin{equation}
\int_{I^{\tilde{n}}} \left( F_i^{\tilde{n}} ,\lambda^{\tilde{n}} \right)_\Gamma \, dt 
= \int_{I^{\tilde{n}}} \left(  b_{i,1} u_{\Gamma ,1}^{\tilde{n}} + b_{i,2} u_{\Gamma ,2}^{\tilde{n}} -  g_i, \lambda^{\tilde{n}} \right)_\Gamma \, dt, 
\quad \forall \lambda^{\tilde{n}} \in \mathcal{V}^{\tilde{n}} (r_i) . 
    \label{eq:DG_1d} 
\end{equation}

Together with initial conditions, complete monolithic algorithms can now be derived 
by choosing substepping methods, described as~\DGIT methods, and then coupling as 
shown above.  As is typical for~\DGIT derivations, quadrature may need to be applied 
to arrive at certain methods.  We illustrate here by providing a complete, multirate 
Crank-Nicolson scheme.  Notationally, given any data $g^{n-1}$ and $g^n$ we define 
\[
g^{n-1/2} := \frac{1}{2} \left( g^{n-1}+g^{n}\right) ,  
\] 
whereas for functional data of the form $g(t^{n-1})$ and $g(t^n)$ we define 
\[
(g)^{n-1/2} := \frac{1}{2} \left( g (t^{n-1})+g (t^{n}) \right) .
\]
The Crank-Nicolson scheme is then given as follows.  On a coupling window, given $U_i^0$ we compute 
$u_i^n$ for $1\leq n\leq M_i$ and $i=1,2$ from 
\begin{align}
       &\left( U_i^n, v_i^n \right) = 
   \left( U_i^{n-1}, v_i^n \right) 
   -\Delta t_i L \left( U_i^{n-1/2}, v_i^n \right) 
   \label{eq:DG_2a_CN} \\ 
   &\qquad \qquad -\Delta t_i\left(  \left( F_i^{\tilde{n}} \right)^{n-1/2} , v_i^n \right)_\Gamma +\Delta t_i \left( \left( f_i \right)^{n-1/2} ,v_i^n \right), \ \forall v_i^n \in \mathcal{U}_i^n (0), \nonumber \\ 
   &\int_{I^{\tilde{n}}} \left( u_{\Gamma ,i}^{\tilde{n}} ,\lambda^{\tilde{n}} \right)_\Gamma \, dt = 
\Delta t_i \sum_{n=1}^{M_i} \left( U_{i}^{n-1/2}  , \left( \lambda^{\tilde{n}} \right)^{n-1/2}\right)_\Gamma  ,
\quad \forall \lambda^{\tilde{n}} \in \mathcal{V}^{\tilde{n}} (r_i) , \label{eq:DG_2b_CN} \\ 
&\int_{I^{\tilde{n}}} \left( F_i^{\tilde{n}} ,\lambda^{\tilde{n}} \right)_\Gamma \, dt 
= \int_{I^{\tilde{n}}} \left(  b_{i,1} u_{\Gamma ,1}^{\tilde{n}} + b_{i,2} u_{\Gamma ,2}^{\tilde{n}} , \lambda^{\tilde{n}} \right)_\Gamma \, dt \label{eq:DG_2c_CN} \\ 
&\qquad \qquad -\Delta t_i \sum_{n=1}^{M_i} \left( \left( g_{i} \right)^{n-1/2}  , \left( \lambda^{\tilde{n}} \right)^{n-1/2}\right)_\Gamma , 
\quad \forall \lambda^{\tilde{n}} \in \mathcal{V}^{\tilde{n}} (r_i) . \nonumber 
\end{align}
Here, we have applied the identity (see~\secref{sec:DGITtime}) 
\[
\left( u_i^n \right)^{n-1/2} = U_i^{n-1/2} .  
\]
The auxiliary variables $u_{\Gamma ,i}^{\tilde{n}}$ could be algebraically eliminated 
using~\eqref{eq:DG_2b_CN} and inserting into~\eqref{eq:DG_2c_CN}, so these need not be 
explicitly computed and stored.  Here,~\eqref{eq:DG_2a_CN} comes from applying trapezoidal 
quadrature to the integral terms in~\eqref{eq:DG_1a_CN}.  Once the substepping method is 
established, the same quadrature is applied to the right side of~\eqref{eq:DG_1c} to 
arrive at~\eqref{eq:DG_2b_CN}.  The quadrature is applied also for the $g_i$-terms 
in~\eqref{eq:DG_1d} to derive~\eqref{eq:DG_2c_CN}, just for consistency 
with the classical treatment of such data when using Crank-Nicolson.  

In the remainder of this section we proceed to demonstrate that methods derived in the proposed 
framework exhibit the correct numerical mimicry of some properties of the governing system.  
In general, quadrature would need to be applied in a way that preserves the desired 
properties when deriving certain methods.  We will not address quadrature in the remainder 
of this paper, except to illustrate the effects by continuing to use Crank-Nicolson as our 
example algorithm.  

\subsection{Discrete flux conservation}\label{sec:flux_cons}  
\begin{defn}[Strong flux conservation]\label{defn:sfc}  
The model problem~\eqref{eq:pdes}-\eqref{eq:ics} has {\it strong flux conservation} if 
\[
\FOne +\FTwo =0 \quad \text{a.e. on}\ \Gamma \times (0,\tf ] . 
\]
\end{defn} 
Often, this property is found in the form $\FOne=\FTwo$, but in our case the fluxes are oriented 
outward relative to their respective subdomains, hence oppositely oriented along $\Gamma$, so the 
condition becomes $\FOne=-\FTwo$.  Consider that with $f_i=0$ and $g_i=0$, if we 
integrate~\eqref{eq:pdes} over $\Omega_i$, apply Green's theorem and the boundary conditions, 
then sum over $i$, we find that the strong flux conservation property yields 
\[
\frac{d}{dt} \sum_{i=1,2} \int_{\Omega_i} u_i \, d\bChar{x} 
= \sum_{i=1,2} \int_{\Gamma_i} \nu_i \bChar{n}_i \cdot \nabla u_i \, d\Gamma_i 
\]
The time evolution of the quantity $\sum_{i=1,2} \int_{\Omega_i} u_i \, d\bChar{x}$  
should only depend on the fluxes through the external boundaries.  On a coupling 
window, we can equivalently say that 
\[
 \sum_{i=1,2} \int_{\Omega_i} u_i \left( t^{\tilde{n}} \right) \, d\bChar{x} 
= 
\sum_{i=1,2} \int_{\Omega_i} u_i \left( t^{\tilde{n}-1} \right) \, d\bChar{x} 
+\int_{t^{\tilde{n}-1}}^{t^{\tilde{n}}}\sum_{i=1,2} \int_{\Gamma_i} \nu_i \bChar{n}_i \cdot \nabla u_i \, d\Gamma_i \, dt , 
\]
where this latter property still holds with the weak, discrete flux conservation assumption  
\[
\int_{t^{\tilde{n}-1}}^{t^{\tilde{n}}} \int_\Gamma F_1 + F_2 \, d\Gamma \, dt = 0,  
\] 

In the context of the multirate~\DGIT methods, we have the following result.  
\begin{thm}[\DGIT flux conservation]\label{thm:conservation}  
Let ${b}_{1,1}=-{b}_{2,1}$, 
${b}_{1,2}=-{b}_{2,2}$ and ${g}_1=-{g}_2$ (then 
the model problem has strong flux conservation).  Then fluxes computed 
in the multirate~\DGIT framework will satify 
\begin{equation}
    F_1^{\tilde{n}} +F_2^{\tilde{n}} =0 \quad \text{(strong)}, 
    \label{eq:DG_strong_conservation} 
\end{equation}
if $r_1=r_2$.  Otherwise, 
\begin{equation}
    \int_{I^{\tilde{n}}} \left( F_1^{\tilde{n}}+F_2^{\tilde{n}} ,\lambda^{\tilde{n}} \right)_\Gamma \, dt 
= 0 , \quad \forall \lambda^{\tilde{n}} \in \mathcal{V}^{\tilde{n}} (s),
     \quad \text{(weak)} 
    \label{eq:DG_weak_conservation} 
\end{equation} 
where $s=\min \{r_1 ,r_2\}$.  
\end{thm}
\begin{proof}
Under the conditions of the theorem, the results are immediate 
consequences of~\eqref{eq:DG_1c}-\eqref{eq:DG_1d}.  
\end{proof} 
We remark that the weak conservation case can actually be reformulated as strong 
conservation for the projection of the fluxes into $\mathcal{V}^{\tilde{n}} (s)$.  

Quadrature may limit the polynomial orders for discrete flux conservation.  
\begin{cor}[Flux conservation, multirate Crank-Nicolson]\label{cor:CN_flux_cons} 
Let $0\leq r_i \leq 1$ for $i=1,2$.  Then under the conditions of~\thmref{thm:conservation}, 
the discrete fluxes $F_i^{\tilde{n}}$ for 
solutions to~\eqref{eq:DG_2a_CN}-\eqref{eq:DG_2c_CN} satisfy 
\[
F_1^{\tilde{n}}+F_2^{\tilde{n}} =0 
\]
if $r_1=r_2$.  Otherwise,  
\[
\sum_{i=1,2}  \left(   \Delta t_i \sum_{n=1}^{M_i} \left( F_i^{\tilde{n}} \right)^{n-1/2} , \lambda^{\tilde{n}} \right)_\Gamma =0, \quad \forall \lambda^{\tilde{n}}  \in \mathcal{V}^{\tilde{n}} (0) .
\]
\end{cor}
\begin{proof}
An immediate consequence of~\thmref{thm:conservation}, since 
the quadrature for time integration of the relevant terms is exact when $r_i\leq 1$.  
\end{proof}

\subsection{Interfacial energy dissipation}\label{sec:Gamma_E} 
The model interfacial energy is also important.  Coupling methods 
can easily introduce numerical instabilities or more subtle, artificial 
model sensitivities if this behavior is not correct.  
\begin{defn}[Interfacial energy dissipation]\label{defn:interface_energy}  
Assume that the matrix of coupling coefficients 
\[
\mathbb{B} := 
\left[ 
\begin{array}{cc} 
b_{1,1} & b_{1,2} \\ 
b_{2,1} & b_{2,2} 
\end{array} 
\right]
\]
is positive semi-definite.  In the absence of interface forcings, $g_i=0$, the {\it interfacial energy} satisfies 
\[
-\sum_{i=1,2} \Fi  u_i\,  \leq 0 \quad \text{on} \quad \Gamma\times (0,\tf ] . 
\]
\end{defn} 
One multiplies through~\eqref{eq:pdes} by $u_i$ then uses Green's theorem and applies the interface conditions 
to verify the interfacial energy property.  Let the solution energy for $u_i$ be denoted by  
\[
\| u_i \| := \sqrt{\int_{\Omega_i} |u_i|^2 \, d\bChar{x}} . 
\]
With $f_i=0$ and $g_i=0$, the monolithic energy equation is 
\begin{equation}
   \frac{d}{dt} \sum_{i=1,2} \frac{1}{2} \|u_i \|^2 
   = -\sum_{i=1,2} L(u_i , u_i) 
   -\sum_{i=1,2} \int_\Gamma \Fi  u_i\, d\Gamma . 
    \label{eq:energy_equation}
\end{equation} 
Given that the net effect of advection and diffusion terms is dissipative (\eg diffusion-dominated 
or if $\nabla \cdot \bChar{s}_i=0$), then $-L(u_i , u_i)\leq 0$.  
It follows that for a positive semi-definite choice of $\mathbb{B}$, 
\begin{equation*}
\left\{
\begin{aligned}
   \frac{d}{dt} \sum_{i=1,2} \frac{1}{2} \|u_i \|^2 
    = -\sum_{i=1,2} L(u_i , u_i) - \int_\Gamma \bChar{u}^\top \mathbb{B} \bChar{u} \, d\Gamma \leq 0 , \\ 
    \bChar{u} := \left[ 
\begin{array}{c} 
u_{1}  \\ 
u_{2}  
\end{array} 
\right] .
    \end{aligned}\right.
\end{equation*}
As we mentioned for the flux conservation property, the weakened interfacial energy condition 
\[
-\sum_{i=1,2} \int_{I^{\tilde{n}}} ( \Fi , u_i )_\Gamma \, dt \leq 0 
\] 
will suffice to ensure the discrete property 
\[
\sum_{i=1,2} \|u_i (t^{\tilde{n}}) \| \leq \sum_{i=1,2} \|u_i (t^{\tilde{n}-1}) \| .  
\]

\begin{thm}[Multirate interfacial energy]\label{thm:DGIT_energy} 
Assume that $\mathbb{B}$ is positive semi-definite.    
Then solutions to~\eqref{eq:side_conditions}-\eqref{eq:DG_1a},~\eqref{eq:DG_1c}-\eqref{eq:DG_1d} for 
$g_1=0=g_2$ satisfy 
\begin{equation}
    \underset{i=1,2}{\sum} -\sum_{n=1}^{M_i} \int_{I_i^n}\left( F_i^{\tilde{n}} ,u_i^n \right)_\Gamma \, dt \leq 0 .
    \label{eq:DGIT_energy} 
\end{equation} 
\end{thm}
\begin{proof}
Insert $\lambda^{\tilde{n}}=F_i^{\tilde{n}}$ in~\eqref{eq:DG_1c} to reduce 
\begin{equation*}
\underset{i=1,2}{\sum} -\sum_{n=1}^{M_i} \int_{I_i^n}\left( F_i^{\tilde{n}} ,u_i^n \right)_\Gamma \, dt
=
\underset{i=1,2}{\sum} -\int_{I^{\tilde{n}}} \left( F_i^{\tilde{n}} ,u_{\Gamma ,i}^{\tilde{n}} \right)_\Gamma \, dt  .
\end{equation*} 
In turn, choose $\lambda^{\tilde{n}}=u_{\Gamma ,i}^{\tilde{n}}$ in~\eqref{eq:DG_1d}, showing that (with 
$g_i=0$) 
\begin{multline*}
\underset{i=1,2}{\sum} -\int_{I^{\tilde{n}}} \left( F_i^{\tilde{n}} ,u_{\Gamma ,i}^{\tilde{n}} \right)_\Gamma \, dt  
   = \underset{i=1,2}{\sum} -\int_{I^{\tilde{n}}} \left( F_i^{\tilde{n}} , b_{i,1} u_{\Gamma ,1}^{\tilde{n}} +b_{i,2} u_{\Gamma ,2}^{\tilde{n}}\right)_\Gamma \, dt \\
    =-\int_{I^{\tilde{n}}}  \int_\Gamma [u_{\Gamma ,1}^{\tilde{n}} \ u_{\Gamma ,2}^{\tilde{n}}]\mathbb{B}\left[ \begin{array}{c} 
    u_{\Gamma ,1}^{\tilde{n}} \\ 
    u_{\Gamma ,2}^{\tilde{n}} 
    \end{array}\right] \, d\Gamma \, dt \leq 0.
\end{multline*}
\end{proof}

The classical energy analysis for Crank-Nicolson begins by selecting 
$v_i^n = U_i^{n-1/2}$ in~\eqref{eq:DG_2a_CN}, motivating the following result 
for the coupling terms.  
\begin{cor}[Interfacial energy, multirate Crank-Nicolson]\label{cor:CN_GammaE} 
Under the conditions of~\thmref{thm:DGIT_energy}, 
solutions to~\eqref{eq:DG_2a_CN}-\eqref{eq:DG_2c_CN} satisfy 
\[
-\sum_{i=1,2} \Delta t_i \left(  \left( F_i^{\tilde{n}} \right)^{n-1/2} , U_i^{n-1/2} \right)_\Gamma \leq 0 .
\]
\end{cor}
\begin{proof}
Select $\lambda^{\tilde{n}} =F_i^{\tilde{n}}$ in~\eqref{eq:DG_2b_CN} and 
$\lambda^{\tilde{n}} =u_{\Gamma ,i}^{\tilde{n}}$ in~\eqref{eq:DG_2c_CN}.  The 
result follows as shown in the proof of~\thmref{thm:DGIT_energy}.
\end{proof}
In the case that the bilinear form $L(\cdot , \cdot)$ is coercive, and with standard 
assumptions on the boundedness of the problem data,~\corref{cor:CN_GammaE} also 
implies existence of solutions and unconditional stability for the multirate 
Crank-Nicolson method.  
The analysis is standard, once the energy equations are summed over the coupling 
window, and we skip the details for brevity.   

\section{Analysis of the monolithic framework}\label{sec:analysis} 
The analysis is presented for the general multirate~\DGIT
framework to show that existence and convergence (as time 
step sizes decrease) will hold for 
all methods derived using the proposed methodology.  However, this level of generality 
includes methods with time step restrictions, which 
therefore appear in the analysis.  The properties of 
specific methods derived through the proposed framework will 
be affected by choices of side conditions and polynomial 
basis, for example, as well as any additional quadrature 
approximations for the integral terms.  Quadrature is 
not addressed here, for brevity.  As shown for the 
multirate Crank-Nicolson example earlier, an individual 
analysis for a particular algorithm can verify stronger 
properties, such as removing time step restrictions.  

\subsection{Mathematical preliminaries}\label{sec:prelims} 
The finite element spaces satisfy $U_i \subset H^1 (\Omega_i)$, 
$i=1,2$, and 
$U_\Gamma \subset H^{1/2}_{00}(\Gamma) \subset L^2 (\Gamma)$, where 
$H^{1/2}_{00}(\Gamma)$ is the classical Lions-Magenes trace space along 
$\Gamma$~\cite{LM1972}.  In~\secref{sec:semidiscrete}, some  
bilinear forms were introduced that are also the inner-products 
for $L^2 (\Omega_i)$ and $L^2 (\Gamma)$; these induce the 
respective norms $\| u_i\| :=(u_i , u_i)^{1/2}$, $i=1,2$, and 
$\| \mu \|_\Gamma := (\mu , \mu )_\Gamma^{1/2}$.  The inner-products 
on the spaces $H^k (\Omega_i)$ are denoted by $(u_i , u_i)_k$ 
for $k\in\nats$, with norm $\|u_i\|_k :=(u_i , u_i)_k^{1/2}$.  
The Euclidean norm for vectors (or absolute value, for scalars) is 
denoted by $|\cdot |$. 

We assume the following coercivity and continuity properties: there exist positive 
constants $L_{i,1}$ and $L_{i,2}$ such that 
\begin{equation}
L_{i,1} \| u_i \|_1^2 \leq L(u_i , u_i)  
\leq L_{i,2} \| u_i \|_1^2 ,\quad \forall u_i \in U_i , \ i=1,2.
    \label{eq:coercivity} 
\end{equation} 
This can be shown to hold, for example, if $\bChar{s}_i\in\mathcal{C}^1 (\overline{\Omega}_i ;\reals^d)$ with 
$\nabla \cdot \bChar{s}_i$ sufficiently small by applying 
standard integral identities (recall that we assume 
$\bChar{s}_i \cdot \bChar{n}_i=0$ on $\partial \Omega_i$) and a Poincar\`e 
inequality.  See,~\eg,~\cite{BS2008,Ciarlet2002} for details.  

Let $D^{(k)}v:= \frac{d^k v}{dt^k}$.  
We introduce the spaces 
\begin{align*}
    \mathcal{H}_i^{k} =  \mathcal{H}_i^{k} (n)  &:= \left\{ v_i^n (\bChar{x},t)
\, : \, 
     \| D^{(m)}v_i^n \| \in L^2 \left({I_i^{n}}\right), 
 \, 0,\leq m\leq k \right\} \\ 
    \mathcal{H}_\Gamma^k = \mathcal{H}_\Gamma^k ({\tilde{n}}) &:= \left\{ \mu^{\tilde{n}} (\bChar{x},t)
\, : \, 
     \| D^{(m)}\mu^{\tilde{n}} \|_\Gamma \in L^2 \left({I^{\tilde{n}}}\right), 
 \, 0,\leq m\leq k 
    \right\} .
\end{align*}
In correspondence, we define space-time norms 
\begin{equation*}
    \begin{aligned}
    \llVert v_i^n \rrVert_k &:= \sqrt{\int_{I_i^n} \sum_{m=0}^k \| D^{(m)}v_i^n \|^2 \, dt}, \quad \forall v_i^n \in \mathcal{H}_i^{k} , \\ 
        \llVert \mu^{\tilde{n}} \rrVert_{\Gamma ,k} &:= \sqrt{\int_{I^{\tilde{n}}} \sum_{m=0}^k \| D^{(m)}\mu^{\tilde{n}} \|^2 \, dt}, \quad \forall \mu^{\tilde{n}} \in \mathcal{H}_\Gamma^{k} .
    \end{aligned}
\end{equation*} 
We will suppress the $k$-subscript in case $k=0$ for these norms.  

\begin{defn}\label{defn:L2_projection}
Let $-\infty <a<b<\infty$ and set $S:=(a,b)$.  The
$L^2$-projector from
$L^2 (S)$ onto $\mathbb{P}_k (S)$ is denoted by
$\mathcal{P}_S^k$.
\end{defn}
\begin{defn}
Let a function $v_i$, or $\lambda$, have the form 
\[
v_i = \sum_{j=1}^{d_{\Omega_i}} v_{i,j}(t) \phi_{i,j}(\bChar{x}), 
\quad \text{or} \quad 
\lambda = \sum_{j=1}^{d_\Gamma} \lambda_j (t) \mu_j (\bChar{x})  , 
\]
where $v_{i,j}\in L^2 (a,b)$, $1\leq j\leq d_{\Omega_i}$, or 
$\lambda_j \in L^2 (a,b)$, $1\leq j\leq d_\Gamma$.  Given 
$S\subset (a,b)$ measurable, we
extend the definition of $\mathcal{P}_S^k$ as follows:
\[
\mathcal{P}_S^k v_i := \sum_{j=1}^{d_{\Omega_i}} \mathcal{P}_S^k \left[ v_{ij}(t)|_S\right] \phi_{ij}(\bChar{x}), \quad \text{or} \quad 
\mathcal{P}_S^k \lambda := \sum_{j=1}^{d_{\Gamma}} \mathcal{P}_S^k \left[ \lambda_{j}(t)|_S\right] \mu_{j}(\bChar{x}) .
\]
\end{defn} 

The next definition is used later in~\lemref{lem:side_conditions} to show that, 
if reasonable side conditions are specified, then polynomials in $\mathbb{P}_q$ can 
be characterized by their projections into $\mathbb{P}_{q-n_s}$ 
plus the side conditions.  This is a key concept for~\DGIT analysis.  
\begin{defn}\label{defn:isomorphism}
Let $n_s$ satisfy $1\leq n_s \leq q+1$, and let strictly
increasing values $\theta_k$ be given for $1\leq k\leq n_s$ with
$-\infty < \theta_1 < \ldots  < \theta_{n_s} \leq 1$.
Denote by $\mathcal{W}_i$ the $n_s$-times product space
\[
\mathcal{W}_i :=U_i \times \ldots \times U_i .  
\]
Denote the times for the side conditions 
(see~\eqref{eq:side_conditions}) by 
$t_i^{n,k} := t_i^{n-1} +\theta_k \Delta t_i$, for 
$1\leq k\leq n_s$.  
We define mappings
$\mathcal{J}_i^{n}: \mathcal{U}_i^n (q)\to \mathcal{U}_i^n (q-n_s)\times \mathcal{W}_i$
for $1\leq n_s <q+1$ by 
\begin{equation*}
\mathcal{J}_i^{n} (v_i^n) := 
\left(\mathcal{P}_{I_i^{n}}^{q-n_s}(v_i^n) , v_i^n |_{t={t_{i}^{n,1}}} , \ldots , v_i^n |_{t={t_{i}^{n,n_s}}} \right), \ 
\forall v_i^n \in \mathcal{U}_i^n (q) .
\end{equation*}
In case $n_s=q+1$,  
$\mathcal{J}_i^{n}: \mathcal{U}_i^n (q) \to \mathcal{W}_i$ and
\begin{equation*} 
\mathcal{J}_i^{n} (v_i^n) := 
( v_i^n |_{t={t_{i}^{n,1}}}, \ldots , v_i^n |_{t={t_{i}^{n,n_s}}} ), \ 
\forall v_i^n \in  \mathcal{U}_i^n (q).  
\end{equation*}
Furthermore, we define the following norm on $\mathcal{U}_i^n (q)$:
\begin{equation*} 
\llVert \mathcal{J}_i^{n} (v_i^n) \rrVert := 
\sqrt{ \llVert \mathcal{P}_{I_i^{n}}^{q-n_s}(v_i^n) \rrVert^2 + \Delta t_i \sum_{k=1}^{n_s} \left\| v_i^n |_{t=t_{i}^{n,k}} \right\|^2} .
\end{equation*}
In case $n_s=q+1$, define $\mathcal{P}_{I_i^{n}}^{-1}:=\{\bChar{0}\}$.
\end{defn}

\subsection{Existence of solutions}\label{sec:existence} 
The next lemmas are used to identify viable choices of
side conditions and to incorporate them into the analyses of
existence and uniqueness.  They are modest extensions of some
results
in~\cite{DelDub1986} to allow for values $\theta_j$ inside the
range $(0,1)$, though we make no claim of novelty.
\begin{lem}\label{lem:side_conditions}
Let $n_s$ satisfy $1\leq n_s \leq q+1$, and let strictly
increasing values $\theta_j$ be given for $1\leq j\leq n_s$ with
$-\infty < \theta_1 < \ldots  < \theta_{n_s} \leq 1$.  Denote by
$\{ \psi_j\}_{j=0}^q$ a set of
polynomials such that $\psi_j \in \mathbb{P}_j (\reals)$ for
$0\leq j\leq q$, and their restrictions
$\hat{\psi}_j:= \psi_j |_{[-1,1]}$ form an $L^2$-orthogonal basis of
$\mathbb{P}_q [-1,1]$.  Define $\tilde{\mathbb{D}}$ to be the matrix
of size $n_s\times n_s$ with entries
\begin{equation} 
( \tilde{\mathbb{D}})_{j,k} := \psi_{m_k} (2\theta_j-1), \quad m_k := k+q-n_s,\quad 1\leq k\leq n_s . 
\label{eq:D_tilde} 
\end{equation}
Then the mappings $\mathcal{J}_i^{n}$ are bijections if and only
if $\tilde{\mathbb{D}}$ is nonsingular.  If $\tilde{\mathbb{D}}$ is
nonsingular, there exist constants $C_1>0$ and $C_2>0$ independent
of $\Delta t_i$, $d_{\Omega_i}$ and $v_i\in \mathcal{U}_i^n (q)$ such that
\begin{align}
C_1 \llVert v_i \rrVert
&\leq \llVert \mathcal{J}_i^n (v_i) \rrVert
\leq C_2 \llVert v_i \rrVert \label{eq:J_trip_norm_equiv}
\end{align}
\end{lem}
The proof is shown in~\appref{proof:lem:side_conditions}. 

We show that the problems are well-posed over a coupling window by constructing 
Cauchy sequences.  There are a few steps necessary.  Assume that 
all variables are already given on coupling window $\tilde{k}$ for 
all $\tilde{k}<\tilde{n}$.  
Let $m$ denote the iteration 
index.  Given $u_i^{k}$ and $U_{i}^{k}$, $k\leq 0$, and $\Fi^{\tilde{k}}$ for $\tilde{k}<\tilde{n}$, $i=1,2$, the 
updates for iteration $m$ are computed as follows.  If $n_s>0$, the 
side conditions are 
\begin{equation} 
    u_{i(m)}^{n} (t_{i}^{n,k}) 
    =  
\sum_{l=1}^n \left(\mathbb{D}\right)_{k,l} U_{i(m)}^{n+1-l} 
    +\sum_{l=n+1}^{k_s+1} \left(\mathbb{D}\right)_{k,l} U_i^{n+1-l}
    , \ 
    1\leq k \leq n_s, \label{eq:iter_side_conditions} 
\end{equation} 
for $1\leq n\leq M_i$.  We solve 
\begin{align}
   \left( U_{i(m)}^{n},v_i^{-} \left(t_i^{n}\right) \right) -\int_{I_i^n} \left(u_{i(m)}^n , \dot{v}_i \right) \, dt &=
   \left( U_{i(m)}^{n-1}, v_i^{+} \left(t_i^{n-1}\right) \right) \label{eq:exist_1} \\
   -\int_{I_i^n}L\left( u_{i(m-1)}^n, v_i \right)  +  \left( F_{i(m-1)}^{\tilde{n}} , v_i \right)_\Gamma  &-\left( f_i ,v_i \right) \, dt, \ \forall v_i\in \mathcal{U}_i^n (q+1-{n_s}), \nonumber
\end{align}
for $1\leq n\leq M_i$ and $i=1,2$.  
Note that $U_{i(m)}^{0}:=U_{i}^{0}$ is given, and does not actually depend 
on $m$.  Also, we have 
\begin{equation}
    \int_{I^{\tilde{n}}} \left( u^{(m)}_{\Gamma ,i} , \mu \right)_\Gamma \, dt = \sum_{n=1}^{M_i}\int_{I_i^n}\left( u_{i(m)}^n , \mu \right)_\Gamma \, dt , \ \forall \mu \in \mathcal{V}^{\tilde{n}} ({r_i}), \ i=1,2, \label{eq:exist_2}
\end{equation}
and the numerical flux conditions 
\begin{equation}
\int_{I^{\tilde{n}}} \left(  F_{i(m)}^{\tilde{n}} , \lambda_i \right)_\Gamma \, dt
= \int_{I^{\tilde{n}}} \left( b_{i,1} u^{(m)}_{\Gamma ,1} +b_{i,2} u^{(m)}_{\Gamma ,2} -g_i , \lambda_i \right)_\Gamma \, dt  , \ \forall \lambda_i \in \mathcal{V}^{\tilde{n}} ({r_i}). 
\label{eq:exist_3} 
\end{equation}  
 
\begin{lem}[Existence of iterations]\label{lem:iter_existence} 
If $n_s>0$, assume that $\tilde{\mathbb{D}}$ 
(see~\lemref{lem:side_conditions}) is nonsingular.  
Given any initial guesses for $m=0$, the 
system~\eqref{eq:iter_side_conditions}-\eqref{eq:exist_3} then 
admits a unique sequence of solutions.  
\end{lem} 
The proof is shown in~\appref{proof:lem:iter_existence}.  

\begin{lem}[Lift and inverse inequalities]\label{lem:ineqs} 
Let $\tau_i$ (\secref{sec:semidiscrete}) be a regular family 
of meshes with mesh parameters $h_i$ for $i=1,2$.  Define  
$h:=\max \{ h_1 , h_2\}$.  If the advection fields satisfy 
$\bChar{s}_i\in [W^{1,\infty}(\Omega_i)]^d$, $i=1,2$, then there 
exists a constant $C>0$, independent of $h$, such that 
\begin{align}
|L(v_i , w_i) |
    &\leq C h^{-2} \| v_i \|\, \| w_i\|, \quad \forall v_i,\, w_i\in U_i, \label{eq:inverse_ineq} \\ 
 |  (v_i , w_i)_\Gamma  |
    &\leq Ch^{-1} \|v_i \|\, \|w_j \|, \quad \forall v_i \in U_i,\, w_j\in U_j, \ i,j\in \{1,2\} \label{eq:lift_ineq1} \\ 
        \left| (v_i , {w}_\Gamma )_\Gamma \right| 
    &\leq Ch^{-1/2} \|v_i \|\, \|{w}_\Gamma \|_\Gamma, \quad \forall v_i \in U_i,\, w_\Gamma \in U_\Gamma , \ i\in \{1,2\} . \label{eq:lift_ineq2}
\end{align}
\end{lem}
\begin{proof}
The results follow from standard 
applications of the Poincar\'e inequality and inverse inequalities 
from finite element analysis; \eg~\cite{BS2008,Ciarlet2002}.  
\end{proof}

The next result shows existence of discrete approximations 
for any methods derived in the monolithic, multirate 
framework, as defined 
by substepping methods in the~\DGIT 
form~\eqref{eq:side_conditions}-\eqref{eq:DG_1a} 
with coupling determined 
via~\eqref{eq:DG_1c}-\eqref{eq:DG_1d}. 
\begin{thm}[Multirate existence]\label{thm:existence}
If the assumptions of~\lemref{lem:iter_existence} 
and~\lemref{lem:ineqs} hold, there exists a constant $C>0$ 
such that, if $\Delta t(h^{-2}+h^{-1})\leq C$, then the 
problem~\eqref{eq:side_conditions}-\eqref{eq:DG_1a},~\eqref{eq:DG_1c}-\eqref{eq:DG_1d}  
admits a unique solution.  
\end{thm}  
The proof is shown in~\appref{proof:thm:existence}

\subsection{Convergence to the semi-discrete solution}\label{sec:convergence} 
Given an ODE system 
with a solution that is smooth enough, one expects the order of 
accuracy for a DG time stepping method using polynomials of 
order $q$ (like our spaces $\mathcal{U}_i^n (q)$) and time step size 
$\Delta t$ to be $\Order{\Delta t^{q+1}}$, as measured in the 
$L^2$-norm in time.  However, side-values can have {\it nodal}  
convergence of order 
$\Order{\Delta t^{2q+2-n_s}}$, with an 
additional requirement for the properties of the system.  
Details are shown in~\cite{DelDub1986}.  
In this section, we show the analogous results for the multirate 
coupling framework.  

In the next result, the consistency assumption~\eqref{eq:side_value_assumption} 
for side conditions means that these conditions generate an interpolant in 
time for $u_i$ at each point $t_i^{n,k}$ with order of accuracy 
$\Delta t_i^{q+1}$.  In the Crank-Nicolson example shown earlier, we 
have $q=1$, $n_s=2$, $k_s=1$, $t_i^{n,1}=t_i^{n-1}$, $t_i^{n,2}=t_i^{n}$,  $(\mathbb{D})_{1,1}=(\mathbb{D})_{2,2}=0$ and 
$(\mathbb{D})_{1,2}=(\mathbb{D})_{2,1}=1$.  
In this case,~\eqref{eq:side_value_assumption} holds vacuously 
because 
\begin{multline*}
\Delta t_i \sum_{k=1}^{n_s} \left\|  u_i (t_{i}^{n,k}) -\sum_{l=1}^{k_s+1} (\mathbb{D})_{k,l} u_i \left(t_i^{n+1-l} \right) \right\|^2 \\ 
=\Delta t_i \left( \| u_i (t_i^{n-1}) -u_i (t_i^{n-1}) \|^2 +\| u_i (t_i^{n}) -u_i (t_i^{n}) \|^2  \right) = 0 . 
\end{multline*} 
\begin{thm}[$L^2$-convergence]\label{thm:L2} 
Assume that the variables $u_i$ and $F_{i}$ 
solving~\eqref{eq:semi_1a}-\eqref{eq:semi_1d} 
satisfy $u_i \in \mathcal{H}_i^{q+1} (n)$, $1\leq n\leq M_i$, 
$(u_i)|_\Gamma \in \mathcal{H}_\Gamma^{r_i+1} (\tilde{n})$ 
and $F_i \in \mathcal{H}_\Gamma^{r_i+1} (\tilde{n})$ for 
all coupling windows $1\leq \tilde{n}\leq N$, where 
$N= (\tf / \Delta t)\in \nats$, with uniform bounds
\begin{equation}
\begin{aligned}
   \sum_{\tilde{n}=1}^N \sum_{n=1}^{M_i} \sum_{m=0}^{q+1} \int_{I_i^n} \left\| D^{(m)} u_i\right\|^2 \, dt &\leq C \\
     \sum_{\tilde{n}=1}^N \sum_{m=0}^{q+1} \int_{I^{\tilde{n}}} \left\| D^{(m)} u_i\right\|_{\Gamma}^2 \, dt &\leq C \\ 
          \sum_{\tilde{n}=1}^N \sum_{m=0}^{q+1} \int_{I^{\tilde{n}}} \left\| D^{(m)} F_i\right\|_{\Gamma}^2 \, dt &\leq C , 
   \end{aligned}
   \label{eq:solution_bounds} 
\end{equation} 
for some $C>0$ independent of $\Delta t$ or $\Delta t_i$.  
The side conditions 
are assumed to support the consistency bounds 
\begin{multline}
    \Delta t_i \sum_{k=1}^{n_s} \left\|  u_i (t_{i}^{n,k}) -\sum_{l=1}^{k_s+1} (\mathbb{D})_{k,l} u_i \left(t_i^{n+1-l} \right) \right\|^2 \\ 
    \leq C (\Delta t_i)^{2q+2} \int_{t_i^{n-k_s}}^{t_i^n} \sum_{m=0}^{q+1} \left\| D^{(m)} u_i \right\|^2\, dt , 
    \label{eq:side_value_assumption} 
\end{multline}  
for $1\leq n\leq M_i$ on every coupling window.  
Also, we assume the side value 
errors for the initialization data satisfy, for some $p>0$, 
\begin{equation}
\sqrt{\sum_{i=1,2} \left\| u_i (t=0)-U_i^0 \right\|^2} \leq C \Delta t^{p} , \label{eq:init_accuracy1} 
\end{equation} 
at the initial time, and if $N_0>1$ then 
\begin{align}
 \sqrt{\sum_{i=1,2} \sum_{n=1}^{M_i} \left\| u_i (t_i^n) -U_i^{n} \right\|^2} &\leq C \Delta t^{p},  \label{eq:init_accuracy2}  \\ 
 \sqrt{\sum_{i=1,2} \sum_{n=1}^{M_i} \llVert u_i -u_i^n \rrVert^2} &\leq C\Delta t^{p}.  \label{eq:init_accuracy3}
 \end{align}
 on each coupling window $\tilde{n}$, $1\leq \tilde{n} <N_0$.  
If $\Delta t_1$ and $\Delta t_2$ are small enough, then on 
each coupling window $\tilde{n}$ for $N_0 \leq \tilde{n} \leq N$, the error for the~\DGIT methods satisfy 
\begin{equation}
         \left\| u_i (t_i^n) -U_i^{n} \right\| \leq C_h \left( \Delta t^{p} +\left(  \sum_{i=1,2}\Delta t_i^{q+1} + \Delta t^{r_i+1}\right) \right) ,  \label{eq:side_errors} 
        \end{equation} 
        for $1\leq n\leq M_i$ and $i=1,2$, and also  
        \begin{equation}
       \sqrt{\sum_{i=1,2} \sum_{n=1}^{M_i}\llVert u_i -u_i^n \rrVert^2} \leq C_h \left( \Delta t^{p} + \sum_{i=1,2}\left( \Delta t_i^{q+1} + \Delta t^{r_i+1}\right) \right) ,  \label{eq:L2_errors} 
\end{equation}
where $C_h>0$ depends on the solution 
of~\eqref{eq:semi_1a}-\eqref{eq:semi_1d}, $N_0$, 
$\tf$, and possibly $h$, but not $\Delta t_i$, $M_i$ ($i=1,2$) $\Delta t$, or $\tilde{n}$.  
\end{thm}

The proof is shown in Appendix~\ref{proof:thm:L2}.  
In~\eqref{eq:side_errors}-\eqref{eq:L2_errors}, there are three sources of 
errors: approximation of initial data and data on the initialization windows 
$1\leq \tilde{n} <N_0$ (assumed order $\Delta t^p$), time integration on 
each subdomain (order $\Delta t_i^{q+1}$) and the coupling errors 
(order $\Delta t^{r_i+1}$).  

A property of~\DGIT methods is that the convergence rate for side values 
can be higher than the $L^2$-rate indicated by~\thmref{thm:L2}.  In the 
context of the multirate coupling framework, an analogous statement holds 
true at synchronization points.  The 
result is shown using a type of duality argument.  Given the side values 
$U_i^{M_i}$, $i=1,2$, for the final substeps $t_i^{M_i}=t^{\tilde{n}}$ on a 
coupling window with index $\tilde{n}$ fixed but arbitrary, $N_0\leq \tilde{n}\leq N$, 
let $w_i:[0,\tf]\to U_i$ and $\Lambda_{i}:(0,\tf]\to U_{\Gamma}$ be the solution to 
\begin{align}
    (v_i , \dot{w}_i) &= L (v_i , w_i) + (v_i ,\Lambda_{i})_\Gamma , \label{eq:semi_3a} \\ 
&\qquad \qquad \qquad \forall v_i \in U_i,\ 0<t\leq \tf,\ i=1,2, \nonumber \\ 
    (w_i (t^{\tilde{n}}), v_i) &= \left( u_i (t^{\tilde{n}}) - U_i^{M_i} ,v_i \right) ,\ \forall v_i\in U_i, \ i=1,2, \label{eq:semi_3b} 
\end{align} 
with the interface conditions given by  
\begin{equation}
(\Lambda_{i}, \mu_i )_\Gamma = (b_{1,i} w_{1} + b_{2,i} w_{2} , \mu_i )_\Gamma , \ \forall \mu_i \in U_\Gamma ,\ i=1,2.  
\label{eq:semi_3e} 
\end{equation}  

Note the transpositions compared with~\eqref{eq:semi_1a}-\eqref{eq:semi_1d}.  As discussed 
in~\secref{sec:semidiscrete}, this DAE problem 
is equivalent to a well-posed ODE system.  Furthermore, we may 
appeal to the ODE form and immediately extract the following lemma. 
\begin{lem}[Auxiliary variable bounds]\label{lem:aux_bounds} 
Define 
\[
J:= \max \{q+2-n_s , r_1+1 , r_2+1 \} .   
\]
Given sufficiently smooth advection fields $\bChar{s}_i$, $i=1,2$, 
there exists a constant $C>0$ such that 
\begin{align}
   \sum_{i=1,2} \sum_{m=0}^{j} \int_0^{\tf} \left\| D^{(m)} w_i \right\|^2 \, dt &\leq C 
   \sum_{i=1,2} \left\| u_i (t^{\tilde{n}}) - U_i^{M_i} \right\|^2 ,\ 0\leq j \leq J, \label{eq:aux_1} \\ 
   \underset{t\in[0,\tf]}{\max} \sum_{i=1,2} \left\| w_i (t) \right\|^2 &\leq C \sum_{i=1,2} \left\| u_i (t^{\tilde{n}}) - U_i^{M_i} \right\|^2 . \label{eq:aux_2} 
\end{align}
\end{lem}
\begin{proof}
See~\cite{DelDub1986,DHT1981} for details.  
\end{proof}

\begin{thm}[Nodal convergence]\label{thm:nodal} 
Let  
$\hw_{\Gamma ,i}^{\tilde{n}} ,\whL_{i}^{\tilde{n}} \in \mathcal{V}^{\tilde{n}}({r_i})$, $1\leq \tilde{n} \leq N$, 
satisfy 
\begin{equation*}
    \hw_{\Gamma ,i}^{\tilde{n}} := \mathcal{P}_{I^{\tilde{n}}}^{r_i} \left(  {w}_{i}|_{\Gamma} \right) , \quad 
    \whL_{i}^{\tilde{n}} := \mathcal{P}_{I^{\tilde{n}}}^{r_i} \left(  {\Lambda}_{i} \right) 
\end{equation*} 
for $i=1,2$, such that 
(via~\eqref{eq:aux_1} and~\lemref{lem:ineqs}) 
\begin{equation}
\sum_{i=1,2}\int_0^{\tf} \left\| w_{i} - \hw_{\Gamma ,i} \right\|_\Gamma^2 \, dt \leq C\left(\Delta t^{2r_1+2} +\Delta t^{2r_2+2}\right)\sum_{i=1,2}\left\| u_i (t^{\tilde{n}}) - U_i^{M_i} \right\|^2 .
    \label{eq:assumption_1} 
\end{equation}
Also, let $\tw_i\in\mathcal{C}\left( [0,\tf] ,U_i\right)$ 
be globally-continuous interpolating polynomials for $w_i$, $i=1,2$, 
such that $\tw_i\in \mathcal{U}_i^n (q+1-n_s)$ for $1\leq n\leq M_i$ on 
all coupling windows, and 
satisfying (use~\eqref{eq:aux_1}) 
\begin{equation}
\sum_{i=1,2} \sum_{m=0}^1 \int_0^{\tf} \left\| D^{(m)}(w_{i} - \tw_{i}) \right\|^2 \, dt \leq C\Delta t^{2q+2-2n_s} \sum_{i=1,2}\left\| u_i (t^{\tilde{n}}) - U_i^{M_i} \right\|^2 .
    \label{eq:assumption_2} 
\end{equation}
Under the conditions of~\thmref{thm:L2},~\lemref{lem:aux_bounds}, and the initialization assumption 
\begin{multline}
    \left\{ \sum_{i=1,2} \left\| u_i (t^{N_0}) - U_i^{M_i} \right\|^2 \right\}^{1/2} \\ 
    \leq C \left( \Delta t^{q+1-n_s} +\sum_{i=1,2} \Delta t^{r_i+1}\right)\left( \Delta t^p +\Delta t^{q+1} +\sum_{i=1,2} \Delta t^{r_i+1}\right) 
    \label{eq:assumption_3} 
\end{multline} 
(note $U_i^{M_i}$ is the side value at time $t^{N_0}$ in~\eqref{eq:assumption_3} only; 
elsewhere it is for time $t^{\tilde{n}}$),  
the side value errors at the synchronization time $t^{\tilde{n}}$ satisfy 
\begin{multline}
    \left\{ \sum_{i=1,2} \left\| u_i (t^{\tilde{n}}) - U_i^{M_i} \right\|^2 \right\}^{1/2} \\ 
    \leq C \left( \Delta t^{q+1-n_s} +\sum_{i=1,2} \Delta t^{r_i+1}\right)\left( \Delta t^p +\Delta t^{q+1} +\sum_{i=1,2} \Delta t^{r_i+1}\right) , 
\label{eq:sync_rates}  
\end{multline}
for any $\tilde{n}$, $N_0 \leq \tilde{n} \leq N$.  
\end{thm}
The proof is shown in Appendix~\ref{proof:thm:nodal}.  

Consider applying the convergence results~\thmref{thm:L2}-\thmref{thm:nodal} to 
the multirate Crank-Nicolson method discussed earlier.  The terms $\Delta t^p$ 
may be ignored, since no initializations are needed for this one-step method.  
Since $q=1$ and $n_s=2=q+1$, 
one could choose $r_i=1$ for $i=1,2$ and achieve convergence rates of 
$\Delta t^2$ in both $L^2$-norms and at the individual synchronization times.  
However, other methods that have a~\DGIT derivation with $n_s <q+1$ may have 
convergence rates at synchronization times higher than their $L^2$ rates.

\section{Summary discussion and future work}\label{sec:summary} 
A framework was proposed for the derivation of multirate time stepping 
algorithms applied toward interface-coupled dissipative systems.  
The problem is to allow numerical time integrators with different 
step sizes on each of two domains. 
In turn, boundary conditions at the interface require computed 
states to be connected across time levels, which can introduce 
a list of challenges for the design of algorithms.  The proposed 
framework represents the component time integrators as 
discontinuous-Galerkin-in-time (\DGIT) methods.  This approach 
introduces a discrete variational structure in time that is exploited 
to couple the integrators via an auxiliary system of 
interface conditions, posed over common intervals of time we 
call {\it coupling windows}.  These conditions use a polynomial 
least-squares approach to project the discontinuous 
state and flux information on the interface from each domain 
into a smoothed space on the window, where the coupling is 
enforced discretely.  

This approach was shown to provide a way to simultaneously 
address multiple challenges for the design of such multirate 
algorithms.  In principle, methods of any order of consistency 
can be derived that enforce conservation of fluxes passed 
between the domains, in an appropriate discrete sense in time, 
together with the correct behavior of the interfacial energy 
attributed to the governing system.  This latter property is 
desirable to avoid a negative impact on the overall dissipative 
behavior of the global system, which could manifest as noise, 
or artificial sensitivity, or even formal instability.  

We have focused on the monolithic system; its 
presentation and rigorous analysis.  The monolithic 
system is implicitly coupled on each window, so for 
implementation one could consider Schur-complement 
approaches to solve the system; see~\secref{sec:DAEschur} for 
some discussion.  In order to illustrate the analytical principles, 
we have proved the existence of discrete solutions for 
the~\DGIT multirate framework and the order of convergence 
in time in the case of a general class of Robin-type 
interface conditions.  Convergence rates were shown to be 
dependent optimally with respect to the underlying 
choice of~\DGIT time integrators, plus an optimal-order 
error contribution from the coupling relative to the 
coarser scaling of the coupling window.  

Only a first step has been taken to introduce 
the multirate~\DGIT framework.  One intention is to 
provide guiding principles to derive multirate coupled 
algorithms, first by identifying a monolithic method 
with the desired properties.  As we showed with an 
example of monolithic, multirate Crank-Nicolson, 
high-order methods with unconditional stability, flux 
conservation and the correct interfacial energy 
behavior can be constructed in the proposed framework.  
In forthcoming work, we will fill in details for the 
derivation of more example methods based on popular 
underlying time integrators, along with the extension 
to problems with different interface conditions.  
These methods will benefit from rather 
free choices for time integrators, time steps, and coupling window 
sizes as compared with past methods.  

We will also discuss 
various solution strategies with computational examples. 
Schur-complement techniques provide the strongest 
overall stability, but may be too expensive or 
cumbersome for some applications.  More classical 
semi-implicit or explicit partitioned techniques will 
also be discussed for comparison.  These may still 
be derived from the~\DGIT framework using additional 
approximations, but in a way that retains desirable 
properties like order of consistency and flux 
conservation.  
Furthermore, the smoothed coupling space may be chosen to 
reduce the total dimensionality of the coupling terms 
(project piecewise discontinuous on multiple intervals 
to continuous on one larger interval), which might be 
used to reduce coupling costs for partitioned strategies, 
perhaps including communication costs in parallel 
configurations.   The smoothing aspect might also 
be used to mitigate artificial sensitivities.  

Finally, we emphasize the theoretical use of the 
monolithic methods, which serves both as a stepping 
stone toward future partitioned methods, but also 
for analysis of multirate coupling algorithms.  
For example, in~\cite{CD2019} the authors introduced 
a theoretical, monolithic multirate method for the 
purpose of analyzing a partitioned method for 
coupled fluid calculations.  It was shown that 
the partitioned methods may be viewed as 
approximations of the monolithic methods in order to 
understand the properties.

\bibliographystyle{siamplain}
\bibliography{refs}

\appendix
\section{Proof of Lemma~\ref{lem:side_conditions}}
\label{proof:lem:side_conditions}
\begin{proof}
Apply the change of variables
$t=t_i^{n-1} +\Delta t_i (1+\hat{t})/2$ and define
$\sqrt{2}\hat{v}(\hat{t})=\sqrt{\Delta t_i} v_i (t)$.
Note that
\[
\llVert \hat{\mathcal{J}} (\hat{v}) \rrVert
= \llVert \mathcal{J}_i^{n} (v_i) \rrVert
\quad \text{and} \quad 
\left\{ \int_{-1}^1 \| \hat{v} \|^2 \, d\hat{t} \right\}^{1/2}
= \llVert v_i \rrVert, 
\]
where
\[
 \llVert \hat{\mathcal{J}} (\hat{v}) \rrVert
:= 
\left\{ \int_{-1}^1 \left\| \mathcal{P}_{(-1,1)}^{q-n_s}(\hat{v}) \right\|^2 \, d\hat{t} + 2\sum_{j=1}^{n_s} \left\| \hat{v} |_{\hat{t}=2\theta_j -1} \right\|^2  \right\}^{1/2} .
\]
It suffices to restrict our attention
to the case that $I_i^{n}=(-1,1)$, and let
$\hat{\mathcal{J}}:=\mathcal{J}_i^{n}$.  Clearly, the dimensions
of $\mathcal{U}_i^n (q)$ and
$\mathcal{U}_i^n ({q-n_s})\times \mathcal{W}_i$
are the same, so bijectivity of $\hat{\mathcal{J}}$ is equivalent
to showing that the kernel of $\hat{\mathcal{J}}$ is trivial.
Each $\hat{v}$ may be written in the form
$\hat{v} = \sum_{j=0}^q \psi_j \hat{v}_j$, $\hat{v}_j \in U_i$ for $0\leq j\leq q$,
with projection
$\mathcal{P}_{(-1,1)}^{q-n_s} (\hat{v}) =\sum_{j=0}^{q-n_s} \psi_j \hat{v}_j$
($\mathcal{P}_{(-1,1)}^{q-n_s} (\hat{v}) =0$, if $n_s=q+1$).
If $\hat{\mathcal{J}} (\hat{v})=0$ then
$\sum_{j=0}^{q-n_s} \psi_j \hat{v}_j=0$.  Also,
$\sum_{k=q-n_s+1}^{q} \psi_k (2\theta_j-1)\hat{v}_k =0$
for $1\leq j\leq n_s$.  We may conclude that $\hat{v} =0$
if and only if $\tilde{\mathbb{D}}$ (see~\eqref{eq:D_tilde}) is
nonsingular.  It remains to show~\eqref{eq:J_trip_norm_equiv}.
The values of $C_1$ and $C_2$ can be determined from the scalar case
$d_{\Omega_i}=1$; since $\llVert \hat{\mathcal{J}} (\hat{v}) \rrVert$
and $\llVert \hat{v} \rrVert$ derive from inner-products, an
orthonormal basis of $U_i$ can be used to show the constants do not
depend on $d_{\Omega_i}$.
Equivalence of norms on $\mathbb{P}_q (-1,1)$ yields the existence
of the constants, which are therefore also independent of $\Delta t_i$.
\end{proof}
\section{Proof of Lemma~\ref{lem:iter_existence}}
\label{proof:lem:iter_existence}
\begin{proof}
It suffices to show 
that~\eqref{eq:iter_side_conditions}-\eqref{eq:exist_1} represents a 
nonsingular linear system for $u_{i(m)}^{n}$ and 
$U_{i(m)}^{n}$ for any $n$ and $m$, since the remaining variables are then well-defined 
through~\eqref{eq:exist_2}-\eqref{eq:exist_3}.  However, we must prove 
sequentially the cases 
$n=1,2,\ldots , M_i$.  That is, given $m$, assume we have 
shown the existence of $u_{i(m)}^{k}$ and 
$U_{i(m)}^{k}$ for $1\leq k<n$ on coupling window $\tilde{n}$.  
Then we show existence for case $k=n$ must hold as well.  
Together,~\eqref{eq:iter_side_conditions}-\eqref{eq:exist_1} 
provide $(q+2)d_{\Omega_i}$ linear equations to solve for an equal 
number of unknowns, using bases of the spaces $U_i$ and 
$\mathcal{U}_i^n ({q+1-{n_s}})$.  Let us verify that this 
system is nonsingular.  The homogeneous form is 
\begin{align}
    u_{i(m)}^{n} (t_{i}^{n,k}) 
    &=  \left(\mathbb{D}\right)_{k,1} U_{i(m)}^{n} 
    , \ 
    1\leq k \leq n_s, 
   \label{eq:hom_side_conditions} \\
\left( U_{i(m)}^{n},v_i^{-} \left(t_i^{n}\right) \right) -\int_{I_i^n} \left( u_{i(m)}^n , \dot{v}_i \right) \, dt &=0 \label{eq:hom_exist_1}
\end{align} 
for all $v_i\in \mathcal{U}_i^n ({q+1-{n_s}})$.   
First, take $v_i=U_{i(m)}^{n}$ on $I_i^{n}$ 
in~\eqref{eq:hom_exist_1} to see that $U_{i(m)}^{n}=0$.  Notationally, 
if $0\leq n_s \leq q$, let 
\[
\Phi_i^{n} := \mathcal{P}_{I_i^{n}}^{q-n_s} (u_{i(m)}^{n}) .
\] 
In case $n_s=q+1$ just define 
$\Phi_i^{n} :=0$, as per~\defref{defn:isomorphism}. 
By~\lemref{lem:side_conditions} 
and~\eqref{eq:hom_side_conditions}, we will have 
$u_{i(m)}^{n}=0$ if $\Phi_i^{n} =0$.  

If $n_s=q+1$ the proof is done, so now let $0\leq n_s\leq q$.  
Choose $v_i = (t_i^{n}-t)\Phi_i^{n}$ on 
$I_i^{n}$.  Insert 
$v_i$ in~\eqref{eq:hom_exist_1}.  Since 
$\dot{v}_i\in \mathcal{U}_i^n ({q-{n_s}})$, 
\begin{multline*}
    0 =-\int_{I_i^n} \left( u_{i(m)}^{n} , \dot{v}_i \right) \, dt
    =-\int_{I_i^n} \left( \Phi_i^{n} , \dot{v}_i \right) \, dt \\ 
    =\llVert \Phi_i^{n} \rrVert^2 
    -\frac{1}{2} \int_{I_i^{n}} (t_i^{n}-t) \frac{d}{dt} \left\| \Phi_i^{n} \right\|^2 \, dt \\ 
    =\frac{1}{2}\llVert \Phi_i^{n} \rrVert^2 
    +\frac{\Delta t_i}{2} \left\| \Phi_i^{n} (t_i^{n-1}) \right\|^2 .
\end{multline*} 
We see that $\Phi_i^{n}=0$, as required.    
\end{proof}
\section {Proof of Theorem~\ref{thm:existence}}
\label{proof:thm:existence}
\begin{proof}
Define 
$\Delta u_{i(m)}^{n} :=u_{i(m)}^{n}-u_{i(m-1)}^{n}$, 
with analogous 
$\Delta$-notation for the other variables.  Let $C>0$ be a generic 
constant; the value changes throughout the proof.  
Assume existence has already been proved on coupling windows 
$\overline{I^{\tilde{k}}}$ for $1\leq \tilde{k}<\tilde{n}$.  
Let $\tilde{k}=\tilde{n}\geq N_0$ (recall that initial data is 
assumed to be provided already on windows $\tilde{n}\leq N_0$).  
Take the difference between~\eqref{eq:exist_1} in cases $m-1$ and 
$m$, then insert $v_i=\Delta U_{i(m)}^{n}$.    
We then apply~\lemref{lem:ineqs} to show that 
\begin{multline*}
  \left\|\Delta U_{i(m)}^{n} \right\| 
   \leq   \left\|\Delta U_{i(m)}^{n-1} \right\|
   +Ch^{-1/2}\sqrt{\Delta t_i \int_{I_i^n} \left\| \Delta F^{\tilde{n}}_{i(m-1)} \right\|_\Gamma^2 \, dt}   \\ 
   +Ch^{-2}\sqrt{\Delta t_i}  \llVert \Delta u^n_{i(m-1)} \rrVert
\end{multline*}
Since $U_{i(m)}^{0}=U_{i}^{0}$ is $m$-independent, 
$\Delta U_{i(m)}^{0}=0$ and we sum over $n$ to find 
that 
\begin{multline}
    \left\|\Delta U_{i(m)}^{n} \right\| 
   \leq  
   C\sqrt{\Delta t_i} \sum_{k=1}^n h^{-1/2}\sqrt{\int_{I_i^k} \left\| \Delta F^{\tilde{n}}_{i(m-1)} \right\|_\Gamma^2 \, dt} \\
   +C\sqrt{\Delta t_i} \sum_{k=1}^n h^{-2} \llVert \Delta u_{i(m-1)}^k \rrVert  \label{eq:exist_pr1} 
\end{multline}
holds for $1\leq n\leq M_i$.  By analogy with the proof 
of~\lemref{lem:iter_existence}, if $0\leq n_s \leq q$, let 
\[
\Phi_i^{n} := \mathcal{P}_{I_i^{n}}^{q-n_s} (\Delta u_{i(m)}^{n}) , 
\] 
or $\Phi_i^{n} :=0$ if $n_s=q+1$.  Take the 
difference between~\eqref{eq:exist_1} in cases $m-1$ and 
$m$ again, this time inserting 
$v_i=(t_i^{n}-t)\Delta u_{i(m)}^{n}$.  We find 
that 
\begin{multline*}
    \frac{1}{2} \llVert \Phi_i^{n} \rrVert^2 
    +\frac{\Delta t_i}{2} \left\| \Phi_i^{n} (t_i^{n-1}) \right\|^2 
    \leq \Delta t_i \left\| \Phi_i^{n} (t_i^{n-1}) \right\| \left\|\Delta U_{i(m)}^{n-1} \right\| \\ 
    +C\Delta t_i \llVert \Phi_i^{n} \rrVert 
    \left( h^{-1/2}
    \sqrt{\int_{I_i^n} \left\| \Delta F^{\tilde{n}}_{i(m-1)} \right\|_\Gamma^2 \, dt} +h^{-2} \llVert \Delta u_{i(m-1)}^n \rrVert \right)
\end{multline*} 
We apply Young's inequality to subsume the $\Phi$-terms, then 
insert~\eqref{eq:exist_pr1}.  
After some algebra, 
\begin{multline}
 \llVert \Phi_i^{n} \rrVert^2 
    \leq 
    C\, n\, h^{-1}(\Delta t_i)^2 \sum_{k=1}^n \int_{I_i^k} \left\| \Delta F^{\tilde{n}}_{i(m-1)} \right\|_\Gamma^2 \, dt \\
    +Ch^{-4}(n)(\Delta t_i)^2 \sum_{k=1}^n \llVert \Delta u_{i(m-1)}^k \rrVert^2 .
    \label{eq:exist_pr2} 
\end{multline}
Bounds due to side conditions are needed if $n_s>0$.  
Subtract~\eqref{eq:iter_side_conditions} from itself in cases 
$m$ and $m-1$.  It follows that 
\begin{equation}
    \Delta t_i \sum_{r=0}^{n_s} \left\| \Delta u_{i(m)}^{n} (t_{i}^{n,r}) \right\|^2 
    \leq C 
\Delta t_i \sum_{k=1}^n \left\| \Delta U_{i(m)}^{k} \right\|^2 , 
\label{eq:exist_pr3} 
\end{equation} 
where $C$ depends on the (fixed) entries of $\mathbb{D}$.  
Now insert~\eqref{eq:exist_pr1} in~\eqref{eq:exist_pr3}.  Together 
with~\eqref{eq:exist_pr2} and~\eqref{eq:J_trip_norm_equiv}, 
these result imply 
\begin{align*}
  \llVert \Delta u_{i(m)}^n \rrVert^2   &\leq 
    Ch^{-1}(n)(\Delta t_i)^2 \sum_{k=1}^n \int_{I_i^k} \left\| \Delta F^{\tilde{n}}_{i(m-1)} \right\|_\Gamma^2 \, dt \\ 
    &\quad +Ch^{-4}(n)(\Delta t_i)^2 \sum_{k=1}^n \llVert \Delta u_{i(m-1)}^k \rrVert^2 \\ 
    \Rightarrow 
    \sum_{n=1}^{M_i} \llVert \Delta u_{i(m)}^n \rrVert^2   
    &\leq 
    C\Delta t^2 \left\{  h^{-1}\llVert \Delta F_{i(m-1)}^{\tilde{n}} \rrVert_\Gamma^2 
+h^{-4}\sum_{n=1}^{M_i} \llVert \Delta u_{i(m-1)}^n \rrVert^2 \right\} .
\end{align*} 
Now sum over $i=1,2$.  The flux terms are bounded 
using~\eqref{eq:exist_2}-\eqref{eq:exist_3} 
with~\lemref{lem:ineqs}, reducing the result to: 
\begin{equation*}
\sum_{i=1,2} \sum_{n=1}^{M_i} \llVert \Delta u_{i(m)}^n \rrVert^2 
\leq 
    C\Delta t^2 \left(  h^{-2}+h^{-4}\right) \sum_{i=1,2}
    \sum_{n=1}^{M_i} \llVert \Delta u_{i(m-1)}^n \rrVert^2  .
\end{equation*}
In fact, using~\eqref{eq:exist_pr1} and arguments thereafter again, we 
have 
\begin{multline*}
    \sum_{i=1,2} \sum_{n=1}^{M_i} \left\{ \llVert \Delta u_{i(m)}^n \rrVert^2 
     + \Delta t_i \left\| \Delta U_{i(m)}^{n} \right\|^2 \right\} \\ 
     \leq 
     C(h^{-4}+h^{-2})\Delta t^2 \sum_{i=1,2} \sum_{n=1}^{M_i} \llVert \Delta u_{i(m-1)}^n \rrVert^2 \\
     \leq 
     C(h^{-4}+h^{-2})\Delta t^2 \sum_{i=1,2} \sum_{n=1}^{M_i} \left\{ \llVert \Delta u_{i(m-1)}^n \rrVert^2 
     + \Delta t_i \left\| \Delta U_{i(m-1)}^{n} \right\|^2 \right\} .
\end{multline*}
If we choose $\Delta t$ small enough (with scaling 
$\Delta t=\Order{h^{2}}$ for small $h$), then there is some $\epsilon$ such 
that $0\leq \epsilon <1$ and 
\begin{multline*}
    \sum_{i=1,2} \sum_{n=1}^{M_i} \left\{ \llVert \Delta u_{i(m)}^n \rrVert^2 
     + \Delta t_i \left\| \Delta U_{i(m)}^{n} \right\|^2 \right\} \\ 
     \leq 
     \epsilon \sum_{i=1,2} \sum_{n=1}^{M_i} \left\{ \llVert \Delta u_{i(m-1)}^n \rrVert^2 
     + \Delta t_i \left\| \Delta U_{i(m-1)}^{n} \right\|^2 \right\}
\end{multline*} 
for each $m\geq 1$, yielding a Cauchy sequence for the set of 
states and side values on the coupling window.  
Convergence as $m\to\infty$ 
and the solution existence follows from here by standard 
arguments.  Note that the global 
system~\eqref{eq:side_conditions}-\eqref{eq:DG_1a} 
and~\eqref{eq:DG_1c}-\eqref{eq:DG_1d} 
for states, side-values, traces and fluxes can be posed as one 
square, linear, system on the coupling window; the existence 
result shown implies also the uniqueness of 
the solution for this monolithic problem on each coupling window.  
\end{proof}
\section{Proof of Theorem~\ref{thm:L2}}
\label{proof:thm:L2}
\begin{proof}
Some notations are used to denote differences 
between the solution of the semi-discrete 
problem~\eqref{eq:semi_1a}-\eqref{eq:semi_1d} and 
and the~\DGIT problem~\eqref{eq:side_conditions}-\eqref{eq:DG_1a},~\eqref{eq:DG_1c}-\eqref{eq:DG_1d}.  These are also referenced in other proofs.  
\begin{defn}[Error notation]\label{defn:errors} 
 
Given some $\hu_i^n \in \mathcal{U}_i^n ({q})$ with the interpolatory property 
\begin{equation}
    \hu_i^{n} (t_{i}^{n,k}) = u_i (t_{i}^{n,k}), \quad 1\leq k\leq n_s ,\ 1\leq n\leq M_i , 
    \label{eq:hu_interp}  
\end{equation}
for $N_0\leq \tilde{n}\leq N$, 
we will need to bound the following errors, and their components: 
\begin{equation*}
\begin{array}{l}
e_i^n := u_i - u_i^n = (u_i -\hu_i^n )+ (\hu_i^n-u_i^n) = \eta_i^n + \psi_i^n , \\ 
 E_i^{n} := u_i \left(t_i^{n}\right) - U_i^{n}.
\end{array}
\end{equation*}
We use~\eqref{eq:side_conditions} to decompose 
\begin{equation}
\left.\begin{aligned} 
    \psi_i^{n} (t_{i}^{n,k}) 
    &=\tau_{i}^{n,k}
    + \sum_{l=1}^{k_s+1} (\mathbb{D})_{k,l} E_i^{n+1-l} \\ 
    \tau_{i}^{n,k} &:= 
    u_i (t_{i}^{n,k}) -\sum_{l=1}^{k_s+1} (\mathbb{D})_{k,l} u_i \left(t_i^{n+1-l} \right)
    \end{aligned} \right\} .
    \label{eq:side_value_lte}  
\end{equation} 
Also, we define 
$\hu_{\Gamma ,i}^{\tilde{n}} ,\whF_{i}^{\tilde{n}} \in \mathcal{V}^{\tilde{n}}({r_i})$ to 
satisfy 
\begin{equation}
    \hu_{\Gamma ,i}^{\tilde{n}} := \mathcal{P}_{I^{\tilde{n}}}^{r_i} \left(  u_i|_{\Gamma} \right) , \quad 
    \whF_{i}^{\tilde{n}} := \mathcal{P}_{I^{\tilde{n}}}^{r_i} \left(  F_{i}\right) 
\end{equation} 
for $1\leq n\leq N$ and $i=1,2$, along with error components 
\begin{equation*}
\begin{array}{l}
e_{\Gamma ,i}^{\tilde{n}} := (u_i)|_{\Gamma} - u_{\Gamma ,i}^{\tilde{n}} = ((u_i)|_{\Gamma} -\hu_{\Gamma ,i}^{\tilde{n}})+ (\hu_{\Gamma ,i}^{\tilde{n}}-u_{\Gamma ,i}^{\tilde{n}}) = \eta_{\Gamma ,i}^{\tilde{n}} + \psi_{\Gamma ,i}^{\tilde{n}} , \\ 
EF_{i}^{\tilde{n}} := F_{i} - F_{i}^{\tilde{n}} 
= (F_{i} -\whF_{i}^{\tilde{n}})+ (\whF_{i}^{\tilde{n}}-F_{i}^{\tilde{n}}) = \Xi_{i}^{\tilde{n}} + \Psi_{i}^{\tilde{n}}.
\end{array}
\end{equation*}
\end{defn}

We proceed with the proof.  Given $\tilde{n}\geq N_0$, it follows from~\eqref{eq:DG_1a} that 
\begin{align}
    \left( E_i^{n},v_i^{n} \left(t_i^{n}\right) \right) -\int_{I_i^n} \left( e_i^n , \dot{v}^n_i \right)\, dt &=
    \left( E_i^{n-1}, v_i^{n} \left(t_i^{n-1}\right) \right) \label{eq:L2_2} \\
   - \int_{I_i^n} L\left( e_i^n, v_i^n \right)  +  \left( EF_{i}^{\tilde{n}} , v_i^n \right)_\Gamma \, dt , & \ \forall v_i^n \in \mathcal{U}_i^n ({q+1-{n_s}}) , \nonumber
\end{align}
for $1\leq n\leq M_i$.  
If $0\leq n_s\leq q$, define 
\[
\Phi_i^{n} := \mathcal{P}_{I_i^{n}}^{q-n_s} ( \psi_{i}^{n}), 
\] 
or $\Phi_i^{n} :=0$ if $n_s=q+1$.  If 
$0\leq n_s\leq q$, insert 
\[
v_i^{n}=\int_t^{t_i^{n}}\Phi_i^{n}\, dt  
\]
into~\eqref{eq:L2_2} and apply the bounds 
\[
\llVert v_i^n \rrVert \leq \frac{1}{3}\Delta t_i 
\llVert \Phi_i^{n} \rrVert \quad \text{and} \quad 
\left\| v_i^{n}\left(t_i^{n-1}\right) \right\| 
\leq \sqrt{\Delta t_i} 
\llVert \Phi_i^{n} \rrVert,
\]
from which it follows that 
\begin{equation}    
 \llVert \Phi_i^n \rrVert 
    \leq 
    \sqrt{\Delta t_i} \left\|E_{i}^{n-1} \right\| 
    +\llVert \eta_i^n \rrVert 
    +C\Delta t_i\left( h^{-2}\llVert e_i^n \rrVert
    +h^{-1/2} \llVert EF_{i}^{\tilde{n}} \rrVert_\Gamma \right) .
    \label{eq:L2_3} 
\end{equation} 
We seek to bound $\psi_i^{n}$ using~\lemref{lem:side_conditions}, 
which requires bounds on the truncation errors due to side conditions.  
These are provided by applying~\eqref{eq:side_value_lte} and bounding 
the truncation terms $\tau_{i}^{n,k}$ using~\eqref{eq:side_value_assumption}.  After this, apply~\eqref{eq:L2_3} 
so that 
\begin{align*}
   &\llVert \psi_i^n \rrVert^2 
    \leq C \llVert \Phi_i^{n} \rrVert^2 
    +C \Delta t_i \sum_{k=0}^{k_s}\left\|E_{i}^{n-k} \right\|^2 
    \quad 
    +C (\Delta t_i)^{2q+2} \int_{t_i^{n-k_s}}^{t_i^n} \sum_{m=0}^{q+1} \left\| D^{(m)} u_i \right\|^2\, dt  \\ 
&\leq C \Delta t_i \sum_{k=0}^{k_s}\left\|E_{i}^{n,j-k} \right\|^2 
    \quad 
    +C (\Delta t_i)^{2q+2} \int_{t_i^{n-k_s}}^{t_i^n} \sum_{m=0}^{q+1} \left\| D^{(m)} u_i \right\|^2\, dt 
    + C\llVert \eta_i^n \rrVert^2 \\
    &+ C\Delta t_i^2\left( h^{-4}\llVert e_i^n \rrVert^2
    +h^{-1} \int_{I_i^n} \left\| EF_{i}^{\tilde{n}} \right\|_\Gamma^2 \, dt \right)  
\end{align*}
Bounds for $e_i^n$ and $EF_{i}^{\tilde{n}}$ are needed, but 
this must be done on the entire coupling window due to the latter 
(flux) terms.  First, sum over $n$: 
\begin{align}
   \sum_{n=1}^{M_i} \llVert \psi_i^n \rrVert^2 &\leq C \Delta t_i \sum_{n=1-k_s}^{M_i}\left\| E_{i}^{n} \right\|^2 
    +C (\Delta t_i)^{2q+2} \int_{t_i^{1-k_s}}^{t_i^{M_i}} \sum_{m=0}^{q+1} \left\| D^{(m)} u_i \right\|^2\, dt  \nonumber \\
    &+ C\sum_{n=1}^{M_i} \llVert \eta_i^n \rrVert^2 + C\Delta t_i^2\left( h^{-4}\sum_{n=1}^{M_i} \llVert e_i^n \rrVert^2
    +h^{-1} \llVert EF_{i}^{\tilde{n}} \rrVert_\Gamma^2 . \right)  
    \label{eq:L2_4} 
\end{align}
Insert $v_i^{n}=E_i^{n}$ in~\eqref{eq:L2_2} to derive 
\begin{multline}
    \left\| E_i^{n} \right\| \leq 
    \left\| E_i^{n-1} \right\| 
    +C\sqrt{\Delta t_i}\left(  h^{-2}\llVert e_i^n \rrVert 
    + h^{-1/2} \sqrt{\int_{I_i^n} \left\| EF_{i}^{\tilde{n}} \right\|_\Gamma^2 \, dt} \right)  \\
    \leq 
    \left\| E_i^{0} \right\|
      +C\sqrt{\Delta t_i}\sum_{k=1}^n \left(  h^{-2}\llVert e_i^k \rrVert
    + h^{-1/2} \sqrt{\int_{I_i^k} \left\| EF_{i}^{\tilde{n}} \right\|_\Gamma^2 \, dt} \right).  \label{eq:L2_5} 
\end{multline} 
Notationally, for $i=1,2$ we define numbers 
\[
A_i^{\tilde{k}-1,M_i} := \sum_{n=0}^{M_i-1} \left\| E_i^{n} \right\|^2 
\] 
on coupling windows $1 \leq \tilde{k}\leq N_0-1$ (if $N_0>1$), 
and 
\[
A_i^{\tilde{k},M_i} := \left\| E_i^{M_i} \right\|^2 , 
\] 
on coupling windows $N_0 -1\leq \tilde{k} \leq N$.  
Then apply~\eqref{eq:L2_5} to bound 
\begin{multline*}
    \sum_{n=1-k_s}^{M_i}\left\| E_{i}^{n} \right\|^2 = \left\|E_{i}^{M_i} \right\|^2 
    +\sum_{\tilde{k}=N_0}^{\tilde{n}} \sum_{n=0}^{M_i-1} \left\|E_{i}^{n} \right\|^2 
    +\sum_{\tilde{k}=\tilde{n}-N_0}^{N_0-2} A_i^{\tilde{k},M_i} \\ 
    \leq CM_i \sum_{\tilde{k}=N_0-1}^{\tilde{n}-1} A_i^{\tilde{k},M_i} 
    +\sum_{\tilde{k}=\tilde{n}-N_0}^{N_0-2} A_i^{\tilde{k},M_i} 
    +C\Delta t  \left( h^{-4}\sum_{n=1}^{M_i} \llVert e_i^n \rrVert^2
    +h^{-1} \llVert EF_{i}^{\tilde{n}} \rrVert_\Gamma^2  \right) .
\end{multline*} 
Insert this in~\eqref{eq:L2_4} to reduce: 
\begin{multline}    
    \sum_{n=1}^{M_i} \llVert \psi_i^n \rrVert^2 
    \leq 
    C\Delta t \sum_{\tilde{k}=\tilde{n}-N_0}^{\tilde{n}-1}  A_i^{\tilde{k},M_i} 
    +C (\Delta t_i)^{2q+2} \int_{t_i^{1-k_s}}^{t_i^{M_i}} \sum_{m=0}^{q+1} \left\| D^{(m)} u_i \right\|^2\, dt
    \\
     +C\sum_{n=1}^{M_i} \llVert \eta_i^n \rrVert^2
    +C\Delta t \Delta t_i \left( h^{-4} \sum_{n=1}^{M_i} \llVert e_i^n \rrVert^2
    +h^{-1} \llVert EF_i^{\tilde{n}} \rrVert_\Gamma^2  \right)  .
    \label{eq:L2_6}
\end{multline}

The trace and flux errors satisfy 
\begin{equation}
\begin{array}{c} 
\int_{I^{\tilde{n}}} \left( e_{\Gamma ,i}^{\tilde{n}} ,\lambda^{\tilde{n}} \right)_\Gamma \, dt 
= 
\sum_{n=1}^{M_i} \int_{I_i^n} \left( e_{i}^n  ,\lambda^{\tilde{n}} \right)_\Gamma \, dt ,
\quad \forall \lambda^{\tilde{n}} \in \mathcal{V}^{\tilde{n}} (r_i), \\ 
\int_{I^{\tilde{n}}} \left( EF_i^{\tilde{n}} ,\lambda^{\tilde{n}} \right)_\Gamma \, dt 
= \int_{I^{\tilde{n}}} \left(  b_{i,1} e_{\Gamma ,1}^{\tilde{n}} + b_{i,2} e_{\Gamma ,2}^{\tilde{n}}, \lambda^{\tilde{n}} \right)_\Gamma \, dt, 
\quad \forall \lambda^{\tilde{n}} \in \mathcal{V}^{\tilde{n}} (r_i) . 
\end{array} 
\label{eq:L2_8} 
\end{equation}  
We decompose the interface errors and bound by applying~\eqref{eq:L2_8}, 
such that 
\begin{multline}
    \llVert EF_{i}^{\tilde{n}} \rrVert_\Gamma  
\leq 
\llVert \Xi_{i}^{\tilde{n}} \rrVert_\Gamma + \llVert \Psi_{i}^{\tilde{n}} \rrVert_\Gamma  
\leq 
\llVert \Xi_{i}^{\tilde{n}} \rrVert_\Gamma + C\sum_{i=1,2}\llVert e_{\Gamma ,i}^{\tilde{n}} \rrVert_\Gamma
\\ 
\leq 
\llVert \Xi_{i}^{\tilde{n}} \rrVert_\Gamma +
\sum_{i=1,2} \left( 
\llVert \eta_{\Gamma ,i}^{\tilde{n}} \rrVert_\Gamma +
Ch^{-1/2}\sqrt{\sum_{n=1}^{M_i} \int_{I_i^n} \left\| e_{i}^n \right\|_\Gamma^2 \, dt} \right) 
\label{eq:L2_9}
\end{multline}
Insert~\eqref{eq:L2_9} in~\eqref{eq:L2_6} 
for the flux errors.  Then apply the triangle inequality for 
$e_i=\eta_i +\psi_i$ and sum over $i=1,2$.  The 
result implies  
\begin{multline}    
    \sum_{i=1,2}\sum_{n=1}^{M_i}\llVert \psi_i^n \rrVert^2 
    \leq C\Delta t\sum_{i=1,2}\Delta t_i h^{-1}\left( 
    \llVert \Xi_{i}^{\tilde{n}} \rrVert_\Gamma^2
    + \llVert \eta_{\Gamma ,i}^{\tilde{n}} \rrVert_\Gamma^2 \right) \\ 
    +C\Delta t \sum_{i=1,2}\sum_{\tilde{k}=\tilde{n}-N_0}^{\tilde{n}-1}A_i^{\tilde{k},M_i} 
    +C\sum_{i=1,2} (\Delta t_i)^{2q+2} \int_{t_i^{1-k_s}}^{t_i^{M_i}} \sum_{m=0}^{q+1} \left\| D^{(m)} u_i \right\|^2\, dt \\
     +C\sum_{i=1,2}\left(1+\Delta t\Delta t_i(h^{-4}+h^{-2}) \right)\sum_{n=1}^{M_i}\llVert \eta_i^n \rrVert^2 \\ 
     + C\Delta t\sum_{i=1,2}\Delta t_i(h^{-4}+h^{-2})\sum_{n=1}^{M_i}\llVert \psi_i^n \rrVert^2 .
     \label{eq:L2_9b} 
\end{multline}
Interpolatory estimates are needed: 
\begin{align}
    \llVert \eta_i^n \rrVert &\leq C\Delta t_i^{q+1} 
    \sqrt{\sum_{m=0}^{q+1} \int_{I_i^n} \left\| D^{(m)} u_i \right\|^2 \, dt} , \nonumber \\ 
    \llVert \eta_{\Gamma ,i}^{\tilde{n}} \rrVert_\Gamma &\leq C\Delta t^{r_i+1} 
    \sqrt{\sum_{m=0}^{r_i+1} \int_{I^{\tilde{n}}} \left\| D^{(m)} u_i \right\|_{\Gamma}^2 \, dt}, \label{eq:L2_reg} \\ 
    \llVert \Xi_{\Gamma ,i}^{\tilde{n}} \rrVert_\Gamma &\leq C\Delta t^{r_i+1} 
    \sqrt{\sum_{m=0}^{r_i+1} \int_{I^{\tilde{n}}} \left\| D^{(m)} F_i \right\|_{\Gamma}^2 \, dt}. \nonumber 
\end{align} 
Here, $C>0$ is fixed and independent of any discretization parameters.  These 
bounds follow from the solution regularity~\eqref{eq:solution_bounds} using 
standard results for pointwise interpolants and approximations by 
orthogonal polynomials; see~\eg~\cite{BS2008,Ciarlet2002,StoerBulirsch2002}. Note 
that the interpolation points $t_i^{n,k}$ in~\eqref{eq:hu_interp} may be supplemented 
with $q+1-n_s$ more points to uniquely define an appropriate $\hu_i^n\in\mathcal{U}_i^n(q)$ with the desired properties.    

We apply a time step restriction 
\begin{equation}
    C\Delta t\Delta t_i (h^{-4}+h^{-2}+h^{-1}) \leq \frac{1}{2} 
    \label{eq:L2_stepsize} 
\end{equation}
after which the 
$\psi_i^n$-terms on the right side of~\eqref{eq:L2_9b} may be subsumed, and the 
result may be reduced also using~\eqref{eq:L2_reg} to show that   
\begin{multline}    
    \sum_{i=1,2}\sum_{n=1}^{M_i}\llVert \psi_i^n \rrVert^2 
    \leq 
    C\Delta t \sum_{i=1,2}\sum_{\tilde{k}=\tilde{n}-N_0}^{n-1}A_i^{\tilde{k},M_i} \\ 
    +C\sum_{i=1,2} (\Delta t_i)^{2q+2} \int_{t_i^{1-k_s}}^{t_i^{M_i}} \sum_{m=0}^{q+1} \left\| D^{(m)} u_i \right\|^2\, dt \\ 
+ C\sum_{i=1,2}  \Delta t^{2r_i+2} \left( {\sum_{m=0}^{r_i+1}\int_{t^{\tilde{n}-1}}^{t^{\tilde{n}}}\left\| D^{(m)}u_{i} \right\|_{\Gamma}^2\, dt} + {\sum_{m=0}^{r_i+1}\int_{t^{\tilde{n}-1}}^{t^{\tilde{n}}}\left\| D^{(m)}F_{i} \right\|_{\Gamma}^2\, dt} \right) .
    \label{eq:L2_10}
\end{multline} 
Next, from~\eqref{eq:L2_5} with $n=M_i$ we have 
\begin{equation*}
A_{\tilde{n}} := \sum_{i=1,2}   \left\| E_{i}^{M_i} \right\| 
\leq 
   A_{\tilde{n}-1} +C\sqrt{\Delta t}\left( h^{-2}\sqrt{\sum_{n=1}^{M_i}\llVert e_i^{n} \rrVert^2}+ h^{-1/2} \llVert EF_{i}^{\tilde{n}} \rrVert_\Gamma \right) , 
\end{equation*} 
for all $\tilde{n}\geq N_0$.  In case $0\leq \tilde{n}<N_0$, define 
\[
A_{\tilde{n}} := \sqrt{\sum_{i=1,2} A_i^{\tilde{n},M_i}} .
\]
Apply the triangle inequality with $e_i^n=\eta_i^n+\psi_i^n$, 
use~\eqref{eq:L2_reg},~\eqref{eq:L2_9} and~\eqref{eq:L2_10} to derive 
\begin{align}
    &A_{\tilde{n}} \leq A_{\tilde{n}-1}     
    +C_h\Delta t\sum_{\tilde{k}=1}^{N_0}A_{\tilde{n}-\tilde{k}} 
    +C_h \sqrt{\Delta t}B_{\tilde{n}}, \label{eq:L2_11} \\ 
    &B_{\tilde{n}} := \sum_{i=1,2} \Delta t_i^{q+1} \sqrt{\int_{t^{\tilde{n}+1-N_0}}^{t^{\tilde{n}}} \sum_{m=0}^{q+1}\left\| D^{(m)} u_i \right\|^2\, dt} \nonumber \\ 
&\qquad  + \sum_{i=1,2}  \Delta t^{r_i+1} \left( \sqrt{\sum_{m=0}^{r_i+1}\int_{t^{\tilde{n}-1}}^{t^{\tilde{n}}}\left\| D^{(m)}u_{i} \right\|_{\Gamma}^2\, dt} + \sqrt{\sum_{m=0}^{r_i+1}\int_{t^{\tilde{n}-1}}^{t^{\tilde{n}}}\left\| D^{(m)}F_{i} \right\|_{\Gamma}^2\, dt} \right) \nonumber 
\end{align}
for $C_h=\Order{h^{-2}+h^{-1}+h^{-1/2}}$ and $\tilde{n}\geq N_0$.  
The following result is used to bound $A_{\tilde{n}}$.  
\begin{lem}[Sequence bound]\label{lem:sequence} 
Let $\{A_{\tilde{n}}\}_{\tilde{n}}$ and $\{B_{\tilde{n}}\}_{\tilde{n}}$, $1\leq \tilde{n} \leq N$ be sequences of 
non-negative real numbers.  Assume that $C_1$, $C_2$ and $\Delta t$ are 
positive constants, and $N_0\in\nats$.  If 
\[
A_{\tilde{n}} \leq A_{\tilde{n}-1} +C_1\Delta t\sum_{\tilde{k}=1}^{N_0} A_{\tilde{n}-\tilde{k}} 
+C_2\sqrt{\Delta t}B_{\tilde{n}}, \quad N_0 \leq \tilde{n} \leq N, 
\] 
then 
\[
A_{\tilde{n}} \leq \tilde{A} (1+N_0 C_1\Delta t)^{\tilde{n}+1-N_0} 
+C_2\sqrt{\Delta t} \sum_{\tilde{k}=N_0}^{\tilde{n}} B_{\tilde{k}} (1+N_0 C_1\Delta t)^{\tilde{n}-\tilde{k}} ,  
\] 
for $N_0 \leq \tilde{n} \leq N$,where 
$\tilde{A}:= \max\{ A_0 , \ldots , A_{N_0-1} \}$.  
\end{lem}
\begin{proof}
See~\cite{DelDub1986}, Lemma A2.2, for details.  
\end{proof} 
Apply~\lemref{lem:sequence} to~\eqref{eq:L2_11} along 
with~\eqref{eq:init_accuracy1}-\eqref{eq:init_accuracy2}.  
After this, use  
$(1+N_0C_h\Delta t)^{\tilde{n}} \leq e^{C_h N_0{\tilde{n}}\Delta t}\leq e^{N_0C_h\tf}$, as well as  
\begin{multline*} 
\Delta t^{n_1} C_h\sqrt{\Delta t}e^{C_h \tilde{n}\Delta t} \sqrt{\sum_{m=0}^{n_1}\int_{t^{N_0-n_2}}^{t^{\tilde{n}-n_2}}\left\| D^{(m)}\left( \cdot \right) \right\|^2\, dt}  \\ 
\leq \Delta t^{n_1} C_h\sqrt{\tf}e^{C_h \tf}\sqrt{\sum_{m=0}^{n_1}\int_0^{\tf}\left\| D^{(m)}\left( \cdot \right) \right\|^2\, dt} , 
\end{multline*} 
with appropriate values of $n_1$ and $n_2$ chosen term-by-term.  This 
proves~\eqref{eq:side_errors} at all synchronization times.  

We can now apply the bounds at the synchronization times along 
with~\eqref{eq:L2_reg}-\eqref{eq:init_accuracy3}  
and~\eqref{eq:L2_10} to prove~\eqref{eq:L2_errors}; we have 
\begin{multline*}
   \sum_{i=1,2} \sum_{\tilde{n}=1}^N \sum_{n=1}^{M_i} \llVert e_i^n \rrVert^2 
    \leq \sum_{i=1,2} \sum_{\tilde{n}=1}^{N_0-1} \sum_{n=1}^{M_i} \llVert e_i^n \rrVert^2 \\ 
    + 2 \sum_{i=1,2}\sum_{\tilde{n}=N_0}^N \sum_{n=1}^{M_i}\llVert \eta_i^n \rrVert^2 
    +2 \sum_{i=1,2}\sum_{\tilde{n}=N_0}^N \sum_{n=1}^{M_i}\llVert \psi_i^n \rrVert^2 \\ 
    \leq C\Delta t^{2p} +
    C\Delta t \sum_{i=1,2}\sum_{\tilde{k}=0}^{N}A_i^{\tilde{k},M_i} 
      +  C\sum_{i=1,2} \Delta t_i^{2q+2} \sum_{m=0}^{q+1} \int_0^{\tf} \left\| D^{(m)} u_i\right\|^2\, dt \\ 
    +C\sum_{i=1,2} \Delta t^{2r_i+2} \left( 
    \sum_{m=0}^{r_i+1} \int_0^{\tf} \left\| D^{(m)} u_{i} \right\|_{\Gamma}^2\, dt
     +\sum_{m=0}^{r_i+1} \int_0^{\tf} \left\| D^{(m)} F_i\right\|_{\Gamma}^2\, dt \right) \\ 
    \leq C \Delta t^{2p} 
    + C \left( \sum_{i=1,2} \Delta t_i^{2q+2} +\Delta t^{2r_i+2} \right),  
\end{multline*} 
from which~\eqref{eq:L2_errors} follows by taking square-roots 
throughout.  Finally,~\eqref{eq:side_errors} can now be shown for the 
side value errors at times between the synchronization times by 
using~\eqref{eq:L2_5} and applying the above results and analogous arguments.  
\end{proof} 
\section{Proof of Theorem~\ref{thm:nodal}}
\label{proof:thm:nodal}
\begin{proof}
At each synchronization time $t^{\tilde{k}}$, define the temporary variables 
$E_i^{\tilde{k}}:=u_i (t^{\tilde{k}}) - U_i^{M_i}$, 
where $U_i^{M_i}$ is the side value associated with the final time on the 
coupling window with index $\tilde{k}$.  We also define 
\[
E_{\tilde{k}} := \sqrt{\sum_{i=1,2} \left\| E_i^{\tilde{k}} \right\|^2}  
\]
and reuse some other error notations from~\defref{defn:errors}.  
Insert $v_i^n = \tw_i|_{I_i^n}$ in~\eqref{eq:L2_2} for each $n$, then 
sum over $n$ and $\tilde{k}$.  It follows that 
\begin{multline}
(E_{\tilde{n}})^2 
= \sum_{i=1,2} \left( E_i^{N_0-1} ,\tw_i \left( t^{N_0 -1} \right) \right)  \\ 
+\sum_{i=1,2}\sum_{\tilde{k}=N_0}^{\tilde{n}}\sum_{n=1}^{M_i}\int_{I_i^n} \left( e_i^n , \dot{\tw}_i \right)  -L\left( e_i^n , \tw_i \right)\, dt  
-\sum_{i=1,2} \sum_{\tilde{k}=N_0}^{\tilde{n}} \int_{I^{\tilde{k}}}\left( EF_{i}^{\tilde{k}} , \tw_i \right)_\Gamma \, dt. \label{eq:nodal_1}
\end{multline} 
We manipulate first the coupling terms 
in~\eqref{eq:nodal_1}: 
\begin{multline}
\int_{I^{\tilde{k}}}\left( EF_{i}^{\tilde{k}} , \tw_i \right)_\Gamma \, dt 
=\int_{I^{\tilde{k}}}\left( EF_{i}^{\tilde{k}} , \tw_i -w_i \right)_\Gamma \, dt \label{eq:nodal_2} \\ 
+\int_{I^{\tilde{k}}}\left( EF_{i}^{\tilde{k}} , w_{i}-\hw_{\Gamma ,i}^{\tilde{k}} \right)_\Gamma \, dt +\int_{I^{\tilde{k}}}\left( EF_{i}^{\tilde{k}} , \hw_{\Gamma ,i}^{\tilde{k}} \right)_\Gamma \, dt .
\end{multline} 
We apply~\eqref{eq:L2_8},~\eqref{eq:semi_3e} and 
use~\defref{defn:errors} to find 
\begin{multline*}
   \sum_{i=1,2} \int_{I^{\tilde{k}}}\left( EF_{i}^{\tilde{k}} , \hw_{\Gamma ,i}^{\tilde{k}} \right)_\Gamma \, dt \\ 
   = \int_{I^{\tilde{k}}}\left( b_{1,1} e_{\Gamma ,1}^{\tilde{k}}+b_{1,2} e_{\Gamma ,2}^{\tilde{k}} , \hw_{\Gamma ,1}^{\tilde{k}} \right)_\Gamma \, dt +\int_{I^{\tilde{k}}}\left( b_{2,1} e_{\Gamma ,1}^{\tilde{k}}+b_{2,2} e_{\Gamma ,2}^{\tilde{k}} , \hw_{\Gamma ,2}^{\tilde{k}} \right)_\Gamma \, dt \\
    = \int_{I^{\tilde{k}}}\left( b_{1,1} \hw_{\Gamma ,1}^{\tilde{k}}+b_{2,1}\hw_{\Gamma ,2}^{\tilde{k}} , e_{\Gamma ,1}^{\tilde{k}} \right)_\Gamma \, dt +\int_{I^{\tilde{k}}}\left( b_{1,2} \hw_{\Gamma ,1}^{\tilde{k}}+b_{2,2}\hw_{\Gamma ,2}^{\tilde{k}} , e_{\Gamma ,2}^{\tilde{k}} \right)_\Gamma \, dt \\ 
    = \sum_{i=1,2} 
    \int_{I^{\tilde{k}}}\left( \whL_{i}^{\tilde{k}} , \psi_{\Gamma ,i}^{\tilde{k}} \right)_\Gamma \, dt 
    = \sum_{i=1,2} \sum_{n=1}^{M_i} \int_{I_i^n}\left( \whL_{i}^{\tilde{k}} , e_i^n \right)_\Gamma \, dt  \\ 
     = \sum_{i=1,2} \sum_{n=1}^{M_i} \int_{I_i^n}\left( \whL_{i}^{\tilde{k}} -\Lambda_{i}, e_i^n \right)_\Gamma \, dt +\sum_{i=1,2} \sum_{n=1}^{M_i} \int_{I_i^n}\left( \Lambda_{i} , e_i^n \right)_\Gamma \, dt 
\end{multline*} 
Insert this into~\eqref{eq:nodal_2}, the result of which is then 
inserted into~\eqref{eq:nodal_1}.  We add and subtract $w_i$ 
as needed and apply~\eqref{eq:semi_3a} to obtain 
\begin{equation*}
\begin{aligned}
(E_{\tilde{n}})^2 
&= \sum_{i=1,2}\left\{ \left( E_i^{N_0-1} ,w_i \left( t^{N_0 -1} \right) \right)  
+ \left( E_i^{N_0-1} , (\tw_i -w_i) \left( t^{N_0 -1} \right)\right) \right\}  \\
&+\sum_{i=1,2}\sum_{\tilde{k}=N_0}^{\tilde{n}}\sum_{n=1}^{M_i} \int_{I_i^n} \left( e_i^n , \dot{\tw}_i -\dot{w}_i \right)   
-L\left( e_i^n , \tw_i -w_i \right) \, dt \\ 
&-\sum_{i=1,2} \sum_{\tilde{k}=N_0}^{\tilde{n}}\int_{I^{\tilde{k}}} \left( EF_{i}^{\tilde{k}}, \tw_i -w_i \right)_\Gamma 
+ \left( EF_{i}^{\tilde{k}}, w_{i}-\hw_{\Gamma ,i}^{\tilde{k}}  \right)_\Gamma \, dt \\ 
&-\sum_{i=1,2} \sum_{\tilde{k}=N_0}^{\tilde{n}} \sum_{n=1}^{M_i} \int_{I_i^n}\left( (\whL_{i}^{\tilde{k}}- \Lambda_{i}),e_i^n  \right)_\Gamma \, dt  .
\end{aligned}
\end{equation*} 
Bounds for these terms are derived by using~\lemref{lem:aux_bounds}, 
along with the Cauchy-Schwarz inequality,~\lemref{lem:ineqs}, then also 
\[
\sqrt{\sum_{i=1,2} \left\| \tw_i\left( t^{N_0 -1} \right) -w_i \left( t^{N_0 -1}\right) \right\|^2}
\leq C (\tf )\sqrt{ \sum_{i=1,2} \int_0^{\tf} \sum_{m=0}^1 \left\| D^{(m)} ( w_i - \tw_i ) \right\|^2 \, dt} 
\]
and~\eqref{eq:semi_3e},~\eqref{eq:assumption_1}, 
plus~\eqref{eq:assumption_2} to achieve 
\begin{multline*}
    (E_{\tilde{n}})^2 
    \leq 
    C\left( 1+\Delta t^{q+1-n_s}\right)E_{N_0-1}E_{\tilde{n}} \\ 
    +C\left( \Delta t^{q+1-n_s} +\Delta t^{r_1+1} +\Delta t^{r_2+1} \right) E_{\tilde{n}} \sqrt{ \sum_{i=1,2}\sum_{\tilde{k}=1}^N \sum_{n=1}^{M_i} \int_{I_i^n}\left\| e_i^n \right\|^2 \, dt} \\ 
    +C\left( \Delta t^{q+1-n_s} +\Delta t^{r_1+1} +\Delta t^{r_2+1} \right) E_{\tilde{n}}\sqrt{ \sum_{i=1,2}\sum_{\tilde{k}=1}^N \int_{I_i^{\tilde{k}}}\left\| EF_i^{\tilde{k}} \right\|^2 \, dt}  .
\end{multline*} 
Divide through by $E_{\tilde{n}}$.  
The final result follows from~\eqref{eq:L2_9},~\thmref{thm:L2}, 
and~\eqref{eq:assumption_3}.  
\end{proof} 

\end{document}